\documentclass[11pt]{article}

\usepackage{geometry}
\usepackage{mathtools}
\usepackage{tikz}
\usetikzlibrary{arrows}
\usetikzlibrary{patterns}
\usetikzlibrary{hobby}
\usetikzlibrary{decorations.pathreplacing}

\usepackage{graphicx}
\usepackage{amsmath,amssymb,amsthm}
\usepackage{enumerate}
\usepackage{hyperref}
\usepackage[active]{srcltx}
\usepackage[T1]{fontenc}
\usepackage{color}
\usepackage[makeroom]{cancel}

\newtheorem{theo}{Theorem}

\newtheorem{lem}{Lemma}[section]
\newtheorem{defi}[lem]{Definition}
\newtheorem{cor}[lem]{Corollary}
\newtheorem{prop}[lem]{Proposition}
\newtheorem{rmk}[lem]{Remark}
\newcounter{asscount}

\newtheorem{ass}[asscount]{Assumption}

\renewcommand{\L}{{\mathcal L}}
\newcommand{\Lt}{{\mathcal L ^*}}
\renewcommand{\t}{{\theta}}
\newcommand{\qtext}[1]{\quad\mbox{#1}\quad}
\newcommand{\qqtext}[1]{\qquad\mbox{#1}\qquad}
\newcommand{\1}{{\boldsymbol 1}}

\newcommand{\rd}{\mathrm d}
\newcommand{\R}{\mathbb{R}}

\newcommand{\grad}{\operatorname{grad}}
\newcommand{\EVI}{\mathsf{EVI}}
\newcommand{\EDI}{\mathsf{EDI}}
\newcommand{\EDE}{\mathsf{EDE}}
\renewcommand{\d}{\mathsf{d}}
\newcommand{\X}{\mathsf{X}}
\newcommand{\E}{\mathsf{E}}
\newcommand{\Dom}{\operatorname{Dom}}
\newcommand{\Vol}{\operatorname{Vol}}
\newcommand{\Ric}{\operatorname{Ric}}
\newcommand{\Hess}{\operatorname{Hess}}

\newcommand{\W}{{\mathcal{W}}}
\newcommand{\F}{{\mathcal{F}}}
\newcommand{\D}{{F}}
\renewcommand{\H}{{\mathcal{H}}}
\newcommand{\Hpi}{{\mathcal{H}_\pi}}

\newcommand{\dist}{\mathsf d}
\newcommand{\dS}{\mathsf d_{\theta}}

 \newcommand{\PP}{{\mathbb{P}}}
 \newcommand{\EE}{{\mathbb{E}}}
 \newcommand{\QQ}{{\mathbb{Q}}}

\newcommand{\mrest}{
  \,\raisebox{-.127ex}{\reflectbox{\rotatebox[origin=br]{-90}{$\lnot$}}}\,%
}
\newcommand{\pf}{{}_{\#}}

\renewcommand{\P}{\mathcal{P}}

\def\dive{\operatorname{div}}

\numberwithin{equation}{section}
\begin{document}
\title{Hidden dissipation and convexity for Kimura equations
}
\date{}
\author{Jean-Baptiste Casteras, L\'eonard Monsaingeon}

\maketitle
\begin{abstract}
In this paper we establish a rigorous gradient flow structure for one-dimensional Kimura equations with respect to some Wasserstein-Shahshahani optimal transport geometry.
This is achieved by first conditioning the underlying stochastic process to non-fixation in order to get rid of singularities on the boundaries, and then studying the conditioned $Q$-process from a more traditional and variational point of view.
In doing so we complete the work initiated in [Chalub \textit{et Al.}, Gradient flow formulations of discrete and continuous evolutionary models: a unifying perspective. Acta App Math., 171(1), 1-50], where the gradient flow was identified only formally.
The approach is based on the \emph{Energy Dissipation Inequality} and \emph{Evolution Variational Inequality} notions of metric gradient flows.
Building up on some convexity of the driving entropy functional, we obtain new contraction estimates and quantitative long-time convergence towards the stationary distribution.
\end{abstract}


\tableofcontents

\section{Introduction}

The Kimura equation
\begin{equation}
\label{eq:Kimura_p_simple}
\partial_t p=\partial^2_{xx}(x(1-x)p)+\partial_x(x(1-x)p\partial_x U)
,\hspace{1cm}t>0, x\in [0,1]
\end{equation}
was derived by Motoo Kimura in the 1950's \cite{kimura1954stochastic,kimura1962probability,kimura1964diffusion} as a large population approximation $N\to\infty$ of discrete Moran-Wright-Fisher models \cite{fisher1923xxi,wright1929evolution,moran1958random}.
Roughly speaking, the latter discrete stochastic processes aim at describing random fluctuations in the genetic expression within a fixed-size but large population of $N\gg 1$ individuals.
Here we consider two alleles $\mathbb A,\mathbb B$ only, and in \eqref{eq:Kimura_p_simple} the variable $x\in [0,1]$ describes the continuous fraction of the focal allele $\mathbb A$ in the population (and $1-x$ that of $\mathbb B$).
The function $p(t,x)$ thus gives the probability of finding a relative composition $x\in [0,1]$ of gene $\mathbb A$ within the population at time $t$.
The smooth potential $U(x)$ encodes selective preference (fitness difference) of $\mathbb A$ over $\mathbb B$ in the spirit of Darwin's evolution.

We cannot stress enough the fact that \emph{no boundary conditions are enforced} in \eqref{eq:Kimura_p_simple}.
Indeed it is well-known in evolutionary genetics \cite{star2013effects} that, once genetic drift starts, variants tend to either undergo extinction or fixate in finite time, respectively $x=0$ ($\mathbb B$ fixates) and $x=1$ ($\mathbb A$ fixates).
As a consequence one expects that ``solutions'' of \eqref{eq:Kimura_p_simple}, if any, should develop singular Dirac-deltas at $x=0,1$ in finite time, corresponding to such extinction/fixation scenarios, and one therefore cannot expect that biologically relevant solutions remain smooth.
This is reminiscent from the discrete Moran-Wright-Fisher models (with state space $i\in\{0,\dots, N\}$ counting the number of type $\mathbb A$ individuals), for which $i=0$ and $i=N$ are \emph{absorbing states}.
Note also that the diffusion coefficient $\theta(x)=x(1-x)$ in \eqref{eq:Kimura_p_simple} vanishes on the boundary $x\in\{0,1\}$, and, although formally parabolic in the interior, the problem is degenerate at the boundary.
No matter how one wishes to approach it, the mathematical analysis (well-posedness, regularity, etc.) of this type of models is bound to fall outside of the scope of standard parabolic theory, and delicate phenomena must eventually be dealt with.
We refer the reader to \cite{EM,epstein2013degenerate} and references therein for the mathematical theory of such degenerate operators.
\\

In the thermodynamic limit $N\to\infty$, and under the so-called \emph{weak selection} assumption, it was proved in \cite{Chalub_Souza:CMS_2009,chalub2006continuous} that solutions of the discrete Moran-Wright-Fisher process converge towards solutions of \eqref{eq:Kimura_p_simple}.
In order both to establish this discrete-to-continuous convergence and to account for finite-time fixation in the limiting model, the right notion of solutions for \eqref{eq:Kimura_p_simple} is that of \emph{measure-valued, very weak solutions}, see Definition~\ref{def:very_Weak_sol} later on.
These solutions will be of the form
$$
p_t(\rd x)=p^\partial_t(0)\delta_0(\rd x)+ \mathring p_t(x)\rd x+ p^\partial_t(1)\delta_1(\rd x)
$$
and, as measures, are really supported on the whole state-space $\bar E=[0,1]$ and not just $E=(0,1)$.
The transient part $\mathring p_t(x)$ will be a smooth interior solution of \eqref{eq:Kimura_p_simple} \emph{without boundary conditions}, and the boundary part $p^{\partial}_t(x)$ accounts for fixation at $x=0,1$.
Moreover, such solutions automatically satisfy the two conservation laws
\begin{equation}
\label{eq:conservation_laws_p}
\frac d{dt}\int_{\bar E} 1\,p_t(\rd x)=0
\qqtext{and}
\frac{d}{dt}\int_{\bar E} F(x)\,p_t(\rd x)=0.
\end{equation}
The first one is nothing but the conservation of total probabilities.
In the second one, $F(x)$ is the so-called \emph{fixation probability} and encodes the probability that, starting from an initial composition $x$, allele $\mathbb A$ eventually fixates and $\mathbb B$ goes extinct.
Both laws are already present in the discrete models and should therefore naturally be inherited, in the limit $N\to\infty$, as a necessary and biologically relevant conservation.
Very roughly speaking, these two laws together with the eventual emergence of fixated Dirac-deltas (both encoded in our choice of a very weak formulation) act as proxies for the missing boundary conditions and render the parabolic problem well-posed for measure-valued solutions \cite{Chalub_Souza:CMS_2009,danilkina2018conservative}.
It can be shown \cite{Chalub_Souza:CMS_2009,chalub2021gradient} that \eqref{eq:conservation_laws_p} is equivalent to the two boundary kinematic relations
$$
\partial_t  p^\partial_t(x)=\mathring p_t(x), \hspace{1cm}
t>0,\,x\in\{0,1\},
$$
which will clearly appear in the infinitesimal generators of the underlying stochastic process in section~\ref{sec:Kimura_Qprocess}.
It turns out that, although degenerate, the diffusion is still strong enough in the interior so as to force $\mathring p_t(x)\geq c_t>0$ for any $t>0$, all the way to the boundary and in particular for $x=0,1$.
Hence $p^\partial_t(x)>0$ for any $t>0$, and singular Dirac-deltas do develop instantaneously just as suggested from the discrete models.
What is more, fixated states actually attract the whole dynamics in the long run: the interior part $\|\mathring p_t\|_\infty\to 0$ exponentially fast and
$$
p_t\rightharpoonup p_\infty=a_\infty\delta_0+b_\infty\delta_1
\hspace{1cm} \mbox{as }t\to\infty,
$$
where the coefficients $a_\infty,b_\infty\geq 0$ can be computed explicitly from the two conservation laws \eqref{eq:conservation_laws_p} and $p_0$ by solving $a_\infty+b_\infty=1,a_\infty F(0)+b_\infty F(1)=\int_{\bar E}F(x) p_0(\rd x)$.
Similarly, any convex combination $a\delta_0 +(1-a)\delta_1$ yields a stationary (very weak) solution of \eqref{eq:Kimura_p_simple}.
\\

On a slightly different note, it has been known for more than twenty years now \cite{JKO98,otto2001geometry} that Fokker-Planck equations (as well as many nonlinear drift-diffusion-aggregation equations) can be interpreted as gradient-flows of certain natural entropy functionals with respect to Wasserstein optimal transport geometry.
(Here and throughout we use the sign convention opposite to physics, and entropy \emph{decreases} in time.)
We refer to \cite{villani2003topics,OTAM} for a gentle and applied introduction and to \cite{villani_BIG,AGS} for an extended account of the theory.
This provided deep and geometric insights, in particular convergence in the long-time regime towards the stationary distribution of the underlying stochastic process can be reinterpreted as a natural consequence of the \emph{displacement convexity} \cite{mccann1997convexity} of the driving entropy.
This was originally derived in a purely Euclidean framework (the Wasserstein distance $\W$ being built upon the underlying Euclidean cost $c(x,y)=|x-y|_{\R^d}^2$), but \cite{lisini2009nonlinear} later dealt with variable coefficients by viewing $\R^d$ as an intrinsic Riemannian manifold reflecting those coefficients, and consequently adapting the Wasserstein distance.
On the discrete side of things, \cite{maas2011gradient} recently tackled the case of finite Markov processes by constructing a suitable discrete Wasserstein-like transportation distance $\mathbb W_N$, and again identifying the discrete Markov dynamics as the gradient-flow of the (discrete) entropy in this finite-dimensional geometry.
N.~Gigli and J.~Maas \cite{gigli2013gromov} proved that, as the discretization $N\to\infty$, the discrete Wasserstein space Gromov-Hausdorff converges towards the standard continuous Wasserstein space (at least on the torus).
Thus, one may expect that gradient-flows converge accordingly.

This strongly suggests that the convergence of discrete Moran-Wright-Fisher models towards continuous Kimura equations should admit a deeper interpretation than just pointwise convergence of solutions, backing up even more the hope that continuous models can serve as a good approximation of the discrete ones.
In some sense, not only ``solutions converge to solutions'', but in fact ``variational gradient structures converge to variational gradient structures''.
A natural notion arising at this point is that of \emph{evolutionary Gamma-convergence} of gradient flows \cite{sandier2004gamma}, which encodes the compatibility between the discrete and continuous gradient flow structures as $N\to\infty$ (roughly speaking, such a convergence amounts to passing to the limit in the \emph{Energy Dissipation Inequality}, as implemented e.g. in \cite{disser2015gradient} for the standard Euclidean Fokker-Planck equation).
This observation was precisely the starting point of \cite{chalub2021gradient}, where the discrete and continuous gradient structures as well as their compatibility were investigated.
In order to take into account the degeneracy at the boundaries, the underlying state space $\bar E=[0,1]$ is not viewed from the usual Euclidean standpoint, but must rather be endowed with the so-called \emph{Shahshahani distance} $\dS$ \cite{shahshahani1979new} much in the spirit of \cite{lisini2009nonlinear}.
The Shahshahani distance is well-known in evolutionary genetics, particularly in the so-called \emph{replicator dynamics} \cite{hofbauer1998evolutionary}, and it is thus no surprise that it arises here.
However, the underlying metric space $(\bar E,\dS)$ is not truly speaking a smooth Riemannian manifold:
some singularities arise at the endpoints $x=0,1$, which is of course reminiscent from the fixation phenomenon, and this poses serious technical issues.
Moreover, and more importantly, the lack of irreducibility in the discrete models (a crucial assumption in \cite{maas2011gradient}) and the resulting degeneracy in the continuous diffusions prevent from directly applying the existing optimal transport theory.

In \cite{chalub2021gradient} it was suggested that, in order to make sense of this variational convergence of gradient flows, one should actually condition both the discrete and continuous underlying stochastic processes to non-fixation in order to first obtain proper gradient-flows, and that only those conditioned gradient flows converge one to another.
In fact we shall make a case later on, at least in the continuous realm, that the unconditioned Kimura diffusion cannot be a gradient flow in any sense, at least not of the natural entropy functional with respect to the biologically relevant Wasserstein-Shahshahani distance.
It turns out that this type of conditioning is known in the probabilistic literature as a \emph{Yaglom-limit}.
The resulting \emph{Quasi-Stationary Distribution} (QSD) and limiting \emph{$Q$-process} were studied in a very abstract framework in \cite{CV}, see also \cite{meleard2012quasi} for a more population-dynamics oriented perspective.
We will heavily rely here on this abstract probabilistic analysis in order to obtain sufficient analytical information about the PDE, in particular on the boundary behaviour of solutions.
We should also point out that, although the discrete model $N<\infty$ actually enjoys a rigorous Riemannian interpretation \cite{maas2011gradient} allowing to truly identify the Moran-Wright-Fisher model as a smooth gradient flow, the continuous limit is more delicate and the identification of the (conditioned) Kimura PDE as a Wasserstein-Shahshahani gradient flow was only formal in \cite{chalub2021gradient}.
\\

The goal of this work is threefold:
\begin{itemize}
\item
First, we provide a detailed construction of the conditioned $Q$-process, which in \cite{chalub2021gradient} was merely exploited as a purely analytical change of variables and not really addressed from a probabilistic standpoint.
This is Theorem~\ref{theo:Qprocess}.
 \item 
 Second, we fill a gap left in \cite{chalub2021gradient} by assigning a rigorous gradient flow structure to the conditioned Kimura equation, in the sense of the Italian school of abstract metric gradient flows \cite{AGS} and as stated in Corollary~\ref{cor:EDI_t0_q} and Theorem~\ref{theo:q_EVI}
 \item
 Third, we use this gradient-flow structure to derive quantitative contraction estimates and long-time convergence, both measured in the intrinsic and biologically relevant Wasserstein-Shahshahani transportation distance.
 The corresponding statement is Theorem~\ref{theo:q_EVI}.
\end{itemize}

In our opinion the main interest of this work is to somehow recast the problematic unconditioned Kimura equation in a more or less standard optimal transport (although the Wasserstein-Shahshahani still encompasses the singularities at the boundary), which then opens the way to well-established machinery.
In this optimal transport setting, one notion usually playing an important role is that of displacement convexity of the driving entropy functional, and we will establish this convexity as well.
From abstract metric theory \cite{AGS}, such convexity usually provides Logarithmic Sobolev Inequalities (LSI in short) and in turn explicit convergence rates towards the stationary distribution.
This convexity approach to contractivity and long-time convergence is well-known \cite{carrillo2003kinetic} and we do not pretend to be extremely original here, although it should be pointed out that our logarithmic Sobolev inequality will include a degenerate weight.
One of the biologically interesting features of our analysis of the conditioned dynamics is that our construction merely relies on the lowest spectral part of the unconditioned operator, i.e. the principal eigenvalue $\lambda_0>0$ and eigenfunctions.
By contrast, usual spectral theory is expected to predict (optimal) convergence rates in the conditioned model as well, but this should rely also on the next spectral level $\lambda_1>\lambda_0$.
In the same spirit, nonlinear logarithmic Sobolev inequalities are known to be stronger than linear Poincar\'e inequalities and associated spectral gaps $\lambda_0>0$, but still convey less information than the next order gap $\lambda_1-\lambda_0>0$.
Equivalently, our analysis can be considered as a finer nonlinear layer, constructed above the zero-th order spectral information solely and not requiring any knowledge of the next modes.
\\

The rest of the paper is organized as follows:
In section~\ref{sec:Kimura_not_GF} we briefly explain why the unconditioned evolution simply cannot be realized as a Wasserstein gradient flow, thus justifying the need to rather look into the conditioned process.
The probabilistic construction of the $Q$-process is carried out in section~\ref{sec:Kimura_Qprocess}, where we also investigate basic properties of the two associated Fokker-Planck equations (unconditioned and conditioned).
In section~\ref{sec:metric} we study some features of our Wasserstein-Shahshahani space, compute the derivative (\emph{metric slope}) of the entropy functional, and crucially establish its displacement convexity.
Our last section~\ref{sec:relate_PDE_metric} finally identifies the conditioned Kimura PDE as a rigorous metric gradient flow, and as a consequence we deduce from the abstract theory a contraction estimates, a weighted logarithmic Sobolev inequality, and the quantitative exponential convergence towards the unique stationary distribution as $t\to+\infty$.

\section{Kimura is NOT a gradient flow}
\label{sec:Kimura_not_GF}
Throughout the rest of the paper, and although we shall exclusively work in dimension one $x\in \bar E=[0,1]$, we choose to still write $\dive,\nabla,\Delta$ instead of simply $\partial_x,\partial^2_{xx}$ to ease the comprehension.

Before jumping into the analysis, let us briefly explain why the Kimura equation cannot be a gradient flow with respect to the Wasserstein distance.
(We defer to Section~\ref{sec:metric} the precise definition of Wasserstein distances and prefer to keep the discussion informal at this stage.)
Following \cite{lisini2009nonlinear}, Fokker-Planck equations with coefficients
\begin{equation}
\label{eq:Fokker-Planck_Lisini}
\partial_t \rho =\dive(A(x)(\nabla \rho+ \rho\nabla V)), \hspace{1cm}t>0,\,x\in\R^d
\end{equation}
can be interpreted as suitable Wasserstein gradient flows
$$
``\frac{d\rho}{dt}=-\grad_{\W_g} \mathcal S_V(\rho)"
$$
of the entropy functionals
$$
\mathcal S_V(\rho)=\int \rho\log\rho +\int \rho V
$$
by using some adapted weighted Wasserstein distance $\W_g$.
More precisely, writing $g_x\coloneqq A^{-1}(x)$ (a symmetric, positive-definite matrix) and viewing the Euclidean space as a Riemannian manifold $(\R^d,g)$, the \textit{ad-hoc} transportation distance is built upon the corresponding Riemannian distance $c(x,y)=\d_g^2(x,y)$ instead of the usual Euclidean cost $c(x,y)=|x-y|^2_{\R^d}$.
Note that \eqref{eq:Fokker-Planck_Lisini} admits as a unique stationary measure the usual Gibbs distribution $\mu=\frac{1}{\mathcal Z}e^{-V}$, and the PDE and entropy also read
\begin{equation}
\label{eq:Fokker-Planck_Lisini_div}
\partial_t \rho =\dive\left(A(x)\rho\nabla \log\left(\frac{\rho}{\mu}\right)\right),
\hspace{1cm}
\mathcal S_V(\rho)=\H_\mu(\rho)=\int\frac{\rho}{\mu}\log\left(\frac{\rho}{\mu}\right)\mu,
\end{equation}
the relative entropy (Kullback-Leibler divergence) $\H_\mu(\rho)$ of $\rho$ relatively to the reference measure $\mu$.

Trying to fit this framework, and in view of the $A(x)=x(1-x)$ coefficient in \eqref{eq:Kimura_p_simple}, one is immediately tempted to view $E=(0,1)$ as a Riemannian manifold with tensor $g_x=\frac{1}{x(1-x)}$, which is known as the Shahshahani metrics.
It is equally tempting to rewrite \eqref{eq:Kimura_p_simple} in a divergence form similar to \eqref{eq:Fokker-Planck_Lisini_div}: in doing so, algebra inevitably leads to
$$
\partial_t p =\dive\left(x(1-x) p \nabla\log\left(\frac{p}{\frac{e^{-U}}{x(1-x)}}\right)\right),
\hspace{1cm}
\mathcal S_{\hat U}(p)=\int p \log p+\int p \underbrace{[U(x)+\log(x(1-x))]}_{\hat U(x)}.
$$

Here a first significant issue immediately arises: the potential $\hat U(x)=U(x)+\log(x(1-x))$ appearing in what \emph{should} be the right entropy $\mathcal S_{\hat U}$ is unbounded from below, and ``the unique stationary measure $\mu$'', if it existed, would necessarily be
$$
\mu(\rd x)=e^{-\hat U(x)}\rd x=\frac{e^{-U(x)}}{x(1-x)}\rd x
$$
(possibly up to $\frac{1}{\mathcal Z}$ scaling factors).
But, $U$ being smooth, this putative probability measure $\mu$ fails to be integrable at the endpoints $x=0,1$ of the domain.
This cannot be fixed by taking principal parts in any way.

A second issue is that there cannot exist a \emph{unique} stationary distribution: as mentioned earlier any convex combination $a\delta_0+(1-a)\delta_1$ gives a stationary solution of \eqref{eq:Kimura_p_simple}, and this non-uniqueness stems from the failure of irreducibility (already in the discrete models).
This is in sharp contrast with the usual picture for standard diffusions and Fokker-Planck equations, where the stationary distribution is completely determined as the unique minimizer of the entropy.

What is worse, and as discussed earlier, biologically relevant solutions must crucially account for fixation and are therefore necessarily measure-valued, $p=\mathring p + p^\partial$ (a smooth interior density plus a singular boundary term).
For such measures the very meaning of $\mathcal S_{\hat U}(p)$ is unclear due to the presence of atoms at $x=0,1$, where $\hat U=\log 0=-\infty$.
A tempting way to easily assign a meaning to this undetermined integral would be to effectively let it depend on the interior restriction only, and consider instead
$$
\mathring {\mathcal S}_{\hat U}(p)\coloneqq \int \mathring p \log \mathring p+\int\mathring  p\,\hat U,
\hspace{1cm}p=\mathring p+p^\partial.
$$
This is better behaved at least for bounded interior densities $\mathring p$, since $\hat U(x)=U(x)+\log(x(1-x))$ is at least $L^1$.
This is however also a dead end:
Indeed, a minimum requirement for any reasonable notion of gradient flow is that the driving functional should be non-increasing in time.
In our case it is not difficult to cook-up sequences of atomless initial data $p^k_0=\mathring p^k_0$ such that $\mathring {\mathcal S}_{\hat U}(p^k_0)\to-\infty$.
For example the restriction $p^k_0(\rd x)=\frac{1}{\mathcal Z^k}\chi_{E_k}(x)\mu(\rd x)$ of $\mu=\frac{e^{-U(x)}}{x(1-x)}\rd x$ to a sequence of increasing intervals $E_k=(1/k,1-1/k)\subset E$, suitably renormalized, leads to $\mathring {\mathcal S}_{\hat U}(p^k_0)\approx -\log\log k$.
In particular there exists $p_0=\mathring p_0$ with arbitrarily negative initial entropy $\mathring {\mathcal S}_{\hat U}(p_0)=-C<0$.
From \cite{chalub2006continuous,Chalub_Souza:CMS_2009} it is known that the transient mass decays exponentially fast, and more precisely $\|\mathring p_t\|_{L^\infty(E)}=\mathcal O(e^{-\lambda_0 t})$ for some $\lambda_0>0$.
Now, because $\log(x(1-x))\in L^1(0,1)$, it is not difficult to check that
$$
\mathring {\mathcal S}_{\hat V}(p_t)=\int \mathcal O(e^{-\lambda_0 t})\log\left(\mathcal O(e^{-\lambda_0 t})\right) + \mathcal O(e^{-\lambda_0 t}) [U+\log(x(1-x))] \xrightarrow[t\to\infty]{}0.
$$
When started from $\mathring {\mathcal S}_{\hat U}(p_0)=-C<0$, ``entropy'' would eventually have to increase at some point in order to reach $\mathring {\mathcal S}_{\hat U}(p_\infty)=0$, and this clearly rules out the naive gradient flow structure.

Another natural attempt at compensating for singularities was to try and modify even further the underlying geometry (rather than the entropy functional), by taking into account the bulk/interface interactions between $x\in E,x\in \partial E$ directly into a modified Wasserstein-like transportation distance penalizing the mass flux from the interior to the boundary.
This was initially the motivation for \cite{monsaingeon2021new} by the second author, but unfortunately the approach failed to describe the basic Kimura dynamics discussed herein.
\\

Let us also mention at this stage the work \cite{carrillo2022optimal}, where some Wasserstein gradient flow structure claimed for \eqref{eq:Kimura_p_simple} is leveraged in order to devise mass-variable numerical schemes accurately capturing the emergence of fixation $p_t\rightharpoonup a\delta_0+b\delta_1$ as $t\to\infty$ (see also \cite{duan2019numerical} for a related Lagrangian version and gradient structure).
The transportation distance therein is our exact same Wasserstein-Shahshahani metric (to be discussed shortly), and the entropy is exactly $\mathcal S_{\hat U}(p)=\int p\log p +\int p[U+\log(x(1-x))]$.
However, their model is drastically different, in that the PDE is supplemented with no-flux Neumann boundary conditions
$$
[\nabla(x(1-x)p)+x(1-x)p\nabla U]\cdot \nu=0,
\hspace{1cm}x\in\{0,1\}.
$$
On the mathematical level this dramatically changes the behaviour of solutions, and atoms do not develop in finite time anymore (although \cite[Theorem 1.1]{carrillo2022optimal} erroneously appeals to the measure-valued well-posedness from \cite{Chalub_Souza:CMS_2009}, leading to apparent inconsistencies).
As a consequence ``their'' entropy is well-defined, since then in our notations $p=\mathring p$ and therefore $\mathcal S_{\hat U}(p)=\mathring {\mathcal S}_{\hat U}(p)$ makes sense at least for atomless and smooth densities $p(x)\rd x$.
From the modeling perspective, the absence of atoms makes these solutions biologically irrelevant: fixations in finite time are prohibited, whereas they do occur for the discrete Moran-Wright-Fisher processes that they aim at approximating.
However, the dynamics is still attracted towards the same asymptotic behaviour $p_t\rightharpoonup a_\infty\delta_0+b_\infty\delta_1$ as $t\to\infty$, and an intriguing question is to understand how well this artificial, atomless model approximates the real biological one.
Finally, let us mention for the sake of completeness that a related Wright-Fisher model in opinion formation is studied in \cite{furioli2019wright}, yet with different boundary conditions and slightly different setup (the drift is repulsive and nonzero at the boundaries to begin with, and this makes the degeneracy milder in some sense).
In particular long-time convergence is established by entropy methods and weighted logarithmic Sobolev inequalities, but once again the model does not account for fixation (either in finite or infinite time).
\section{The Kimura process and the $Q$-process}
\label{sec:Kimura_Qprocess}
We will actually work with the generalization 
\begin{equation}
\label{eq:Kimura_p}
\partial_tp=\Delta(\t(x) p )+ \dive(\t(x)p\nabla U),
\hspace{1cm}
t>0,\,x\in[0,1]
\end{equation}
of the Kimura equation \eqref{eq:Kimura_p_simple}, for which boundary conditions again cannot be stipulated in order to allow for fixation.
We write $E=(0,1)$, the state space will be $\bar E=[0,1]$, and in the rest of the paper we enforce without further mention
\begin{ass}
\label{ass:theta}
The coefficients
$ \t,U\in C^\infty(\bar E)$, with $\t(x)>0$ in the interior and
$$
\mbox{simple zeros}\quad\t(0)=\t(1)=0
$$
\end{ass}
\begin{ass}
\label{ass:p0}
 The initial probability distribution $p_0\in \P(\bar E)$ is not completely fixated to start with, i.e. $\operatorname{supp} p_0\not\subset \{0,1\}$.
\end{ass}
The strong $C^\infty$ regularity is assumed here for all coefficients in order not to obfuscate the analysis, but can certainly be relaxed.
In Assumption~\ref{ass:theta} the linear vanishing of $\t$ is really critical:
Different vanishing rates $\t(x)\sim \d_\R(x,\partial E)^\nu$ for $\nu\neq 1$ would lead to very different behaviours, in the sense that either the boundaries would end-up at infinite distance of the interior (a particle starting from $x\in (0,1)$ would not be able to reach the boundary in finite time anymore), or Brownian particles would be allowed to jump back in the domain after being fixated for some random but finite time \cite{feller1951diffusion} (whereas linear behaviour here really guarantees eternal fixation).
Assumption~\ref{ass:p0} is enforced with the sole purpose of ruling out trivial scenarios:
For, whenever $p_0=p^\partial_0$ is initially supported on the absorbing states $x=0,1$, then $p_t\equiv p_0$ for later times and no interesting dynamics can unravel.

For convenience we record here the notion of very weak solutions from \cite{Chalub_Souza:CMS_2009}, which is the right setting if one wishes to allow for fixation while maintaining well-posedness.
Note that the test-functions below are not required to vanish on the lateral boundary.
\begin{defi}
\label{def:very_Weak_sol}
Given $p_0\in \P(\bar E)$ we say that a narrowly continuous curve of measures $t\mapsto p_t(\rd x)$ is a \emph{very weak solution} of \eqref{eq:Kimura_p} if
\begin{multline*}
 -\int_0^\infty\int_{\bar E}\partial_t\varphi(t,x)\, p_t(\rd x)\rd t
 -\int_{\bar E}\varphi(0,x)\,p_0(\rd x)
\\
=\int_0^\infty\int_{\bar E}\t(x) [\Delta \varphi(t,x)-\nabla U(x)\cdot \nabla\varphi(t,x)]\, p_t(\rd x)\rd t
\end{multline*}
for every $\varphi\in C^{1,2}_{c}([0,\infty)\times\bar E)$.
\end{defi}

In order to help the reader let us slightly anticipate on our notations:
$(p_t,(\PP_x)_{x\in \bar E},P_t,\L)$ and $(q_t,(\QQ_x)_{x\in E},\tilde P_t,\tilde \L)$ will stand for the density, path-measure, semi-group, and generator of the unconditioned Kimura stochastic process and of the subsequent conditioned $Q$-process, respectively.
The law $p_t\in \P(\bar E)$ of the unconditioned process will have atoms on the boundaries, and we shall always decompose accordingly
$$
 p_t = \mathring p_t+ p_t^\partial \quad \in \P(E)+\P(\partial E)\cong \P(\bar E).
$$
Here we mean of course $\mathring p=p\mrest E$ and $p^\partial=p\mrest\partial E$ in the sense of restrictions of measures.
Typically the interior part $\mathring p_t(\rd x)=\mathring p_t(x)\rd x$ will have smooth density w.r.t. the Lebesgue measure, while we write $p^\partial_t=p_t^\partial(0)\delta_0+p_t^\partial(1)\delta_1$ with a slight abuse of notations.

Let us also anticipate that we shall have to consider several successive potentials: starting from the smooth potential $U$ in \eqref{eq:Kimura_p}, the conditioning will first induce via some $h$-transform an additional logarithmic drift $2\frac{\nabla\eta}{\eta}=2\nabla\log\eta$ (here $\eta(x)$ is some eigenfunction to be defined later on).
We will accordingly write $\tilde U=U-2\log \eta$ for the conditioned potential.
For the conditioned process, we will show that there exists a unique stationary measure $\pi(\rd x)=\frac{1}{\mathcal Z}\frac{e^{-\tilde U(x)}}{\t(x)}\rd x=\frac{1}{\mathcal Z}e^{-(\tilde U+\log\t)(x)}\rd x$, and we will write accordingly $V=\tilde U+\log \t$.
Finally, our Riemannian viewpoint on $E=(0,1)$ will give rise to the intrinsic Shahshahani volume form $\Vol(\rd x)=\operatorname{vol}(x)\rd x=\frac{1}{\sqrt{\t(x)}}\rd x$, and it will be convenient to take into account this geometric effect by considering next $W=V +\log\operatorname{vol}=V- \log\sqrt\t$.
Summarizing, the reader may want to keep the following roadmap at hand:
$$
U\xrightarrow[\mbox{drift }2\frac{\nabla\eta}{\eta}]{\mbox{conditioning}}\tilde U
\xrightarrow[\pi=\frac{1}{\mathcal Z}e^{-V}\rd x]{\mbox{Gibbs}} V
\xrightarrow[\Vol(\rd x)=\frac{1}{\sqrt{\t(x)}}\rd x]{\mbox{geometry }} W
$$
A last potential implicitly appearing in section~\ref{sec:Kimura_not_GF} was $\hat U=U+\log\t$, which arised when naively trying to write \eqref{eq:Kimura_p} in divergence form $\partial_t p=\Delta(\t p)+\dive(p\t\nabla U)=\dive(\t\nabla p)+\dive(p\t\nabla \hat U)$.
A crucial feature of these various potentials will be their boundary behaviours, summarized here for convenience as well
\begin{equation}
\label{eq:behaviour_all_potentials}
\mbox{as }x\to \partial E=\{0,1\}:
\hspace{1cm}
\left\{
\begin{array}{lcl}
 U & \qquad &\mbox{smooth, bounded}\\
 \hat U & & \to -\infty\\
 \tilde U & & \to +\infty\\
 V & & \to +\infty\\
 W & & \to +\infty
\end{array}
\right..
\end{equation}
We wish to emphasize that the probabilistic analysis in this upcoming section is very strongly inspired from \cite{CV}, to say the very least, and as such we tried to stick to the notations thereof as best as possible.
\subsection{Unconditioned dynamics}
From now on we claim full mathematical rigor.
Consider the strong $\bar E$-valued Markov process defined by the Stochastic Differential Equation
\begin{equation}
\label{eq:SDE_Kimura}
\begin{cases}
\rd X_t=-\t(X_t)\nabla U(X_t)\rd t +\sqrt{2\t(X_t)}\rd B_t\\
\mbox{killed on the boundaries }x\in \partial E=\{0,1\}
\end{cases}
\end{equation}
(The $\sqrt 2$ factor is just here to ensure that we end up with the usual PDE convention for the Laplacian operator $\Delta$ in the sequel, and not the probabilistic $\frac 12 \Delta$ instead.)

Standard one-dimensional methods \cite[Section 5.5]{karatzas2012brownian} give well-posedness in the strong sense, and for any initial distribution $p_0\in \P(\bar E)$ the SDE has a pathwise unique global solution.
Notably, the linear vanishing of $\t$ guarantees that $X_t$ reaches the boundary in finite time, almost surely.
We also refer to Feller's work \cite{feller1951diffusion} for related one-dimensional models.
By the It\^o formula it is easy to check that the infinitesimal generator $\L f(x)\coloneqq \lim\limits_{t\to 0^+} \frac{\EE_xf(X_t)-f(x)}{t}$ is
\begin{equation}
\label{eq:def_L}
\L f(x)=
\begin{cases}
 \t(x)\Delta f(x) - \t(x)\nabla U(x)\cdot \nabla f(x) & \mbox{if }x\in E,\\
 0 & \mbox{if }x\in \partial E,
\end{cases}
\end{equation}
at least for $f\in C^2(\bar E)$.
In order to compute the adjoint operator $\Lt$ acting on measures $p\in \P(\bar E)$, recall that $p=\mathring p+p^\partial$ decomposes into an interior part and a boundary part, and observe that
$$
 \int_{\bar E} f(x) (\Lt  p)(\rd x)
 =\int_{\bar E} (\L f)(x) p(\rd x) = \int_{E} (\L f)(x)\mathring p(\rd x)
$$
because $\L f(x)$ vanishes on the boundary.
Next, if $ \mathring p(\rd x)=\mathring p(x)\rd x$ and $f(x)$ are smooth enough, integration by parts readily gives
\begin{multline*}
\int_{E} (\L f) p
=\int_{E} (\L f)\mathring p\, \rd x
=
\int_{E} (\t\Delta f -\t\nabla U\cdot \nabla f) \mathring p\,\rd x
\\
=
\int_{E} [\Delta(\t  \mathring p) +\dive(\t  \mathring p\nabla U)] f \, \rd x
+\int_{\partial E}
[\t \mathring p\nabla f - f\nabla(\t \mathring p) - \t \mathring p f\nabla U]\cdot \nu\,\rd x.
\end{multline*}
(Here and throughout $\nu=\nu_{\mathrm{ext}}$ stands for the outer unit normal.)
Because $\t$ vanishes linearly on $\partial E$ all the boundary terms where derivatives are not applied to $\t$ cancel, and we are left with
$$
\int_{\bar E} f (\Lt  p)
= 
\int_{E} f\,[\Delta(\t  \mathring p) +\dive(\t  \mathring p\nabla U)] \,\rd x
-\int_{\partial E} f\mathring p\nabla \t\cdot \nu\,\rd x
$$
at least for smooth densities $\mathring p$.
In more practical terms with $\partial E=\{0,1\}$, this identifies the dual generator as
$$
\Lt  p =\underbrace{\Delta(\t  \mathring p) +\dive(\t  \mathring p\nabla U) }_{\mbox{interior}}+ \underbrace{ \mathring p(0)\partial_x\t(0)\delta_0 -  \mathring p(1)\partial_x\t(1)\delta_1}_{\mbox{boundary}}.
$$
Note that 
\begin{enumerate}[(i)]
 \item 
 This only depends on the interior part $ \mathring p$ of $ p=\mathring p + p^\partial$
 \item
 The interior and boundary terms consist of, respectively, a usual differential operator and a singular measure
 \item
 The coefficients $\mathring p(0)\partial_x\t(0),-\mathring p(1)\partial_x\t(1)$ on the boundary are nonnegative (owing to Assumption~\ref{ass:theta} $\t$ has positive slope at $x=0$ and negative slope at $x=1$) and depend on the boundary trace $ \mathring p(0), \mathring p(1)$ of the interior part $ \mathring p$ only, but not on the singular part $ p^\partial$ as one might have expected.
\end{enumerate}
As a consequence, and decomposing again $p=\mathring p+p^\partial$, the abstract Fokker-Planck equation ${\partial_t p_t=\Lt p_t}$
for the law $p_t$ of $X_t$ reads in PDE terms as
\begin{equation}
\label{eq:FP_with_BC}
\begin{cases}
 \partial_t \mathring p_t = \Delta(\t\mathring p_t)+\dive(\t \mathring p_t\nabla U)
 \\
\partial_t p^\partial_t(0)=+\partial_x \t(0) \mathring p_t(0)
\\
\partial_t p_t^\partial(1)=-\partial_x \t(1) \mathring p_t(1).
\end{cases}
\end{equation}
This parabolic problem is nonstandard in the sense that no boundary conditions are imposed for the interior part $\mathring p_t$.
In the boundary flux conditions, $\mathring p_t(0), \mathring p_t(1)$ should be understood as the interior limit (boundary trace) of $\mathring p_t(x)$ as $x\to\{0,1\}$.
As already noticed in \cite[pp. 24]{chalub2021gradient} and \cite{Chalub_Souza:CMS_2009} this, together with the total probability $p_t(\bar E)=\mathring p_t(E)+p^\partial_t(\partial E)=1$, is equivalent to the two conservation laws
\begin{equation}
\label{eq:conservation_laws}
\frac{d}{dt}\int_{\bar E} 1\,p_t (\rd x) =0
\qqtext{and}
\frac{d}{dt}\int_{\bar E} F(x)\, p_t(\rd x) =0,
\end{equation}
where $F(x)$ is the so-called \emph{fixation probability} defined by
$$
\begin{cases}
\t(\Delta F -\nabla U\cdot\nabla F)=0 ,
\\
F(0)=0,F(1)=1.
\end{cases}
$$
Biologically speaking, $F(x)$ encodes the probability that, starting from an initial frequency $x\in [0,1]$ of $\mathbb A$ genes in the population, the process $X_t$ gets absorbed at $X_\tau=1$ (and thus the $\mathbb B$ allele eventually goes extinct in finite time).
Note that, here in dimension 1, the pair $(1,F)$ solve the same 2nd order ODE $G''-U'G'=0$ with different boundary conditions.
They are therefore linearly independent and span the kernel of $\L=\t[\Delta  -\nabla U\cdot\nabla ]$.
In a more geometric language, \eqref{eq:conservation_laws} thus reads as $p_t\in (\operatorname{Ker}\L)^\perp$ for all $t>0$.
\\

We collect now some useful properties of the Fokker-Planck equation.
\begin{lem}[adapted from \cite{Chalub_Souza:CMS_2009}]
\label{lem:prop_p_chalub}
Fix an arbitrary $p_0=\mathring p_0+p^{\partial}_0\in\P(\bar E)$.
There exists a unique global solution $p_t=\mathring p_t+p^{\partial}_t\geq 0$ of \eqref{eq:FP_with_BC} in the sense of Definition~\ref{def:very_Weak_sol}.
This solution satisfies the two conservation laws \eqref{eq:conservation_laws}, and $\mathring p_t$ has density $\mathring p_t(x)\in C^\infty_{t,x}(\R^+\times\bar E)$.
One has the expansion
\begin{equation}
 \label{eq:expansion_series_p}
 \mathring p _t(x)=\sum\limits_{j\geq 0} e^{-\lambda _j t}\hat P_j \alpha_j(x),
 \hspace{1cm}
 \hat P_j=\int_E\alpha_j(x)\ \mathring p_0(\rd x) 
\end{equation}
where $0<\lambda_0<\lambda_1<\dots$ and $\alpha_j(x)$ are a suitable spectral basis for $\Lt$.
The principal eigenfunction $\alpha_0(x)\geq c_0>0$ is bounded away from zero and smooth up to the boundary.
\end{lem}

This follows from a closer inspection of results from \cite{Chalub_Souza:CMS_2009} (in particular Lemma 2.2, Lemma 2.3, and Theorem 4.3 therein), and essentially consists in constructing the spectral sequence $(\lambda_j,\alpha_j(x))_{j\geq 0}$ for a certain Sturm-Liouville problem of \emph{limit-circle-non-oscillatory type} \cite{zettl2012sturm}.
The spectral problem is set in a singular weighted space $L^2(\frac{1}{\t(x)}\rd x)$ making the corresponding operator positive-definite (thus resulting in $\alpha_0(x)\geq c_0>0$), but for the sake of brevity we skip the details.
As far as well-posedness is concerned one could equally appeal to \cite{EM}, but since we shall need some spectral information anyway we opted for the constructive approach here based on the expansion \eqref{eq:expansion_series_p}.

Note in particular that $\lambda_1-\lambda_0>0$, and that \eqref{eq:expansion_series_p} provides the interior part $\mathring p_t$ only.
The boundary part $p^\partial_t$ can be reconstructed a posteriori by integrating the boundary conditions in \eqref{eq:FP_with_BC}, namely
\begin{equation}
\label{eq:kinematics_pboundary_from_pint}
p^\partial_t(x)=p^\partial_0(x) -  \partial_\nu\t(x)\int_0^t\mathring p_s(x) \,\rd s
\qqtext{for}x\in\partial E=\{0,1\}.
\end{equation}
It is also worth stressing that, as is apparent from the explicit representation \eqref{eq:expansion_series_p}, the evolution of the interior part $\mathring p_t$ is completely determined by its initial datum $\mathring p_0$, and does not depend on the initial boundary part $p^\partial_0$.

Let us now setup the probabilistic framework needed to define the $Q$-process later on.
We denote $({\bar E}, (\F_t)_{t\geq 0}, (X_t)_{t\geq 0},(P_t)_{t\geq 0}, (\PP_x)_{x\in \bar E})$ the Markov process with state space $\bar E=E\cup\partial E$, associated to the SDE \eqref{eq:SDE_Kimura} starting from $X_0=x$.
This is actually a Feller process.
As usual, when starting from an $X_0$ initially distributed according to some $p_0\in\P(\bar E)$, we write
$$
\PP_{p_0}(\bullet)\coloneqq\int_{\bar E}\PP_x(\bullet)p_0(\rd x)
$$
and the corresponding expectations are denoted $\EE_x, \EE_{p_0}$.
Note that this diffusion process always has continuous paths, hence the absorption time
$$
\tau\coloneqq \tau_{\partial E}=\inf\left\{t\geq 0:\quad X_t\in \partial E\right\}
$$
is a well-defined stopping time.
By definition we have $X_t=X_\tau$ for all $t\geq \tau$.
It is not difficult to check that $\tau<\infty$, $\PP_x$-almost surely for all $x\in \bar E$, and that
$$
\PP_x(t<\tau)>0
\qquad \mbox{for all }x\in E\mbox{ and }t\geq 0.
$$
The following exponential ergodicity of the Kimura process will be crucial for us, by allowing direct applications of abstract results from \cite{CV}
\begin{theo}
\label{theo:exist_alpha}
Let $(\lambda_0,\alpha_0(x))$ be the principal eigenpair as in Lemma~\ref{lem:prop_p_chalub}.
 The probability measure
 $$
 \alpha(\rd x)\coloneqq \frac{1}{\int_E\alpha_0(y)\rd y}\alpha_0(x)\rd x
 $$
 is a \emph{quasi-limiting distribution}, i-e
 \begin{equation}
  \label{eq:expo_CV_alpha}
  \lim\limits_{t\to+\infty}\PP_{p_0}(X_t\in B \,\vert\, \tau>t)=\alpha(B)
 \end{equation}
 exponentially fast in time and uniformly w.r.t. both the initial distribution $p_0\in \P(\bar E)$ and Borel sets $B\in \mathcal B(E)$.
It is also the unique \emph{quasi-stationary distribution} (QSD) \cite{meleard2012quasi} of the Markov process \eqref{eq:SDE_Kimura},
 \begin{equation}
 \label{eq:decay_PPalpha_QSD}
 \PP_\alpha(\bullet\,\vert\,t<\tau)=\alpha(\bullet)
 \qtext{and}
 \PP_\alpha(t<\tau)=e^{-\lambda_0t}
 \qqtext{for all}t\geq 0.
 \end{equation}
 \end{theo}
 It is worth stressing that, as one could have anticipated, the rate $\lambda_0>0$ in \eqref{eq:decay_PPalpha_QSD} is exactly the principal eigenvalue.
 A closer inspection of the proof below will reveal that the exponential convergence rate in \eqref{eq:expo_CV_alpha} is essentially given by the next order gap $\lambda_1-\lambda_0>0$.
 For the existence and uniqueness of QSD's for killed diffusion processes we refer e.g. to \cite{champagnat2018criteria,benaim2021degenerate} and references therein, see also \cite{guillin2020quasi} for more general Feller processes and hypoelliptic diffusions.
 \begin{proof}
Fix $t>0$ and any Borel set $B\in\mathcal B(E)$.
Note that, because $B\cap\partial E=\emptyset$ and by definition of the absorption time $\tau$, $X_t\in B$  implies $t<\tau$ and therefore
 $$
 \PP_{p_0}(X_t\in B\,\vert\,t<\tau)
 =\frac{\PP_{p_0}(\{X_t\in B\}\cap \{t<\tau\})}{\PP_{p_0}(t<\tau)}
 =\frac{\PP_{p_0}(X_t\in B)}{\PP_{p_0}(X_t\not \in \partial E)}
 =\frac{\PP_{p_0}(X_t\in B)}{\PP_{p_0}(X_t\in E)}.
 $$
 Now, since the law $p_t\in \P(\bar E)$ of $X_t$ decomposes as $p_t=\mathring p_t+p^{\partial}_t$ and $B\subset E$, we conclude that
\begin{equation}
\label{eq:recast_proba_QSD}
\PP_{p_0}(X_t\in B\,\vert\,t<\tau)=\frac{\PP_{p_0}(X_t\in B)}{\PP_{p_0}(X_t\in E)}
=\frac{\int_B p_t(\rd x)}{\int_E p_t(\rd x)}
=\frac{\int_B \mathring p_t(\rd x)}{\int_E \mathring p_t(\rd x)}.
\end{equation}
In order to take the limit in this last quotient, we leverage the expansion \eqref{eq:expansion_series_p} of $\mathring p_t$ in order to control
  $$
  \left|\mathring p_t(x)-e^{-\lambda_0t}\hat P_0\alpha_0(x)\right|
  \leq \sum\limits_{j\geq 1}e^{-\lambda_j t}\left|\hat P_j\right|\,\|\alpha_j\|_{\infty}
  $$
  uniformly in $x$.
  Notice that $|\hat P_j|=\left|\int_E\alpha_j(x) \mathring p_0(\rd x)\right|\leq \|\alpha_j\|_\infty \mathring p_0(E)\leq \|\alpha_j\|_\infty$.
  From \cite{Chalub_Souza:CMS_2009} it is known that $\|\alpha_j\|_\infty\lesssim j^{3/4}$ and $\lambda_j \gtrsim j^2$ (only depending on the operator $\Lt$).
  This controls
  $$
  \left|\mathring p_t(x)-e^{-\lambda_0t}\hat P_0\alpha_0(x)\right|\leq e^{-\lambda_1(t-1)}\sum_je^{-C j^2}j^{3/2}\leq Ce^{-\lambda_1 t}
  $$
  for all $t\geq 1$.
  Since all the constants involved at this stage depend only on $\Lt$ through its spectral sequence $(\lambda_j,\alpha_j)$, one has in \eqref{eq:recast_proba_QSD} 
  \begin{multline*}
  \PP_{p_0}(X_t\in B\,\vert\,t<\tau)
  = \frac{e^{-\lambda_0t}\hat P_0 \int_B\alpha_0(x)\rd x + \mathcal O(e^{-\lambda_1 t})}{e^{-\lambda_0t}\hat P_0 \int_E\alpha_0(x)\rd x + \mathcal O(e^{-\lambda_1 t})}
  \\
  =\frac{\int_B\alpha_0(x)\rd x}{\int_E\alpha_0(x)\rd x}
  + \mathcal O\left(|\hat P_0|e^{-(\lambda_1-\lambda_0)t}\right)
  =\alpha(B) + \mathcal O\left(e^{-(\lambda_1-\lambda_0)t}\right),
\end{multline*}
where all the $\mathcal O$ terms are uniform in $p_0, B$.
Here we also used again that $|\hat P_0|\leq \|\alpha_0\|_\infty\leq C$ uniformly in $p_0$ as required.

The fact that this quasi-limiting distribution $\alpha$ is indeed the unique QSD and satisfies \eqref{eq:decay_PPalpha_QSD} readily follows from \cite[Theorem 2.1]{CV} and the proof is complete.
 \end{proof}

An immediate consequence is the existence of a principal eigenpair:
\begin{prop}
\label{prop:exist_eta}
 The QSD $\alpha$ satisfies
 $$
 \Lt\alpha=-\lambda_0\alpha.
 $$
 The limit
 \begin{equation}
  \label{eq:def_eta}
  \eta(x)=\lim\limits_{t\to+\infty}\frac{\PP_x(t<\tau)}{\PP_\alpha(t<\tau)}\geq 0
 \end{equation}
 is finite, converges uniformly in $x\in \bar E$, belongs to the domain $\mathcal D(\L)$, and provides a simple
 eigenfunction
 $$
 \L\eta =-\lambda_0\eta.
 $$
 Moreover $\eta\in C^\infty(\bar E)$ is positive in $E$ and vanishes linearly on the boundary.
\end{prop}
\begin{proof}
First of all, our Theorem~\ref{theo:exist_alpha} allows to apply any result from \cite[section 3]{CV} by validating condition (iv) in Theorem 2.1 therein.
In order to get $\Lt\alpha=-\lambda_0\alpha$, recall that by definition $\alpha$ is simply a renormalization of $\alpha_0$ from Lemma~\ref{lem:prop_p_chalub}, and that the latter was precisely the principal eigenfunction $\Lt \alpha_0=-\lambda_0\alpha_0$ from the Sturm-Liouville spectral analysis in \cite{Chalub_Souza:CMS_2009}.
 This would also have followed from the well-known fact \cite{meleard2012quasi} that a QSD always satisfies $P_t^*\alpha=e^{-\lambda_0 t}\alpha$, which entails the claim by simple differentiation in time at $t=0^+$.
 
 That $\eta>0$ in \eqref{eq:def_eta} belongs to the domain, is an eigenfunction $\L\eta=-\lambda_0\eta$, and vanishes at the boundary follows from \cite[Prop. 2.3]{CV}.
 The multiplicity one is exactly the content of \cite[Cor. 2.4]{CV}.
 
 Observe next that $\eta\in \mathcal D(\L)$ is at least continuous up to the boundary.
 Usual local elliptic regularity for the elliptic equation $\Lt\eta=-\lambda_0\eta$ guarantees $C^\infty$ regularity in the interior.
As for the boundary regularity we only discuss here $x=0^+$, the right boundary $x=1^-$ can be handled in the exact same way.
A close inspection of the asymptotic behaviour for the ODE $\L\eta+\lambda_0\eta=0$ shows that solutions all behave, for $x\to 0^+$, as linear combinations $\eta(x)=C_1 x F(x) + C_2 G(x)$ of a suitable Kummer confluent hypergeometric function $F$ and a Meijer $G$-function $G$.
The former $xF(x)$ is smooth up to the boundary and vanishes linearly, while the latter has non-trivial limit $G(0)\neq 0$.
But then $\eta(0)=0$ implies $C_2=0$.
This simultaneously entails the full boundary regularity and the linear behaviour, and the proof is complete.
\end{proof}
\subsection{Conditionned dynamics}
Roughly speaking, the $Q$-process is obtained by conditioning the previous Kimura process to non-absorption $\tau=+\infty$.
However, since $\tau<\infty$ almost surely, this conditioning is too singular and one needs to work out a little more.
\begin{theo}
\label{theo:Qprocess}
Let $p_0\in \P(\bar E)$ be any initial distribution, and define
\begin{equation}
\label{eq:def_q_0_from_p0}
q_0(\rd x)\coloneqq \frac{\eta(x)p_0(\rd x)}{\int_E \eta(y)p_0(\rd y)}\in \P(E).
\end{equation}
 Then
 \begin{equation}
 \label{eq:def_QQs}
  \QQ_{q_0}(A_s)\coloneqq \lim\limits_{t\to\infty} \PP_{p_0}(A_s \vert t<\tau),
 \hspace{1cm}
 \forall A_s\in\F_s
 \end{equation}
 is a well-defined probability measure, with moreover
 \begin{equation}
 \label{eq:dPP/dQQ}
 \left.\frac{\rd \QQ_{q_0}}{\rd\PP_{p_0}}\right|_{\F_s}
 =\frac{e^{\lambda_0 s}\eta(X_s)}{\int_E \eta(y){p_0}(\rd y)}
 \eqqcolon M_s
 \end{equation}
 The conditioned semi-group is
 \begin{equation}
 \label{eq:Pttilde}
 \tilde P_t f(x)=\frac{e^{\lambda_0t}}{\eta(x)}P_t(\eta f)(x).
 \end{equation}
 Denoting $\QQ_x\coloneqq \QQ_{\delta_x}$ the corresponding measure obtained by taking $p_0=q_0=\delta_x$, the $Q$-process $({\bar E}, (\F_t)_{t\geq 0},(X_t)_{t\geq 0},(\tilde P_t)_{t\geq 0}, (\QQ_x)_{x\in E})$ is an $E$-valued strong Markov process with generator
 \begin{equation}
 \label{eq:def_generator_Ltilde}
 \tilde \L f=\lambda_0 f +\frac{\L(\eta f)}{\eta}.
 \end{equation}
 Finally, for any $s\geq 0$, the law $q_s\in\P(E)$ of $X_s$ under $ \QQ_{q_0}$ is given by
 \begin{equation}
 \label{eq:qs_from_ps}
 q_s(\rd x)=\frac{\eta(x)p_s(\rd x)}{\int_E \eta(y)p_s(\rd y)},
 \end{equation}
 where $p_s\in \P(\bar E)$ denotes the law of $X_s$ under the unconditioned measure $\PP_{p_0}$.
\end{theo}
Note that, since $\eta$ vanishes on the boundary, we have $\eta(x) p_s(\rd x)=\eta(x)\mathring p_s(\rd x)$ and in particular $q_s$ does not charge the boundary $\partial E$.
This justifies writing $q_s\in \P(E)$ instead of $\P(\bar E)$, and is also of course consistent with the fact that our conditioning to ``$\tau=+\infty$'' precisely prevents $X_t$ from hitting the boundary (when viewed from the $ \QQ_{q_0}$ perspective, not from $\PP_{p_0}$).
It is also worth pointing out that, owing to \eqref{eq:qs_from_ps}, the initial distribution of $X_0$ under the conditioned measure is really $q_0$ as defined in \eqref{eq:def_q_0_from_p0}.
This is of course consistent with our notation $\QQ_{q_0}$ in \eqref{eq:def_QQs} as well as the customary relation $\QQ_{q_0}=\int \QQ_x \,q_0(\rd x)$.
\begin{proof}
The proof (and the result) is a minor extension of \cite[Theorem 3.1]{CV}, where the same statement was proved only for $p_0=\delta_x$.
For the reader's convenience we only reproduce the main steps, slightly adapted in order to cover the case of arbitrary initial distributions $p_0$.

 Write $\Gamma_t\coloneqq \1_{t<\tau}$ and let
 $$
 \QQ^{t}_{q_0}\coloneqq \frac{\Gamma_t}{\EE_{p_0}(\Gamma_t)}\PP_{p_0}.
 $$
 We need to show that $\QQ^t_{p_0}$ has a proper limit, whose density with respect to $\PP_{p_0}$ restricted to $\F_s$ is precisely $M_s$.
 By the Markov property, and because by definition $\PP_{p_0}=\int \PP_y\,{p_0}(\rd y)$, we have first for any $t\geq s$
 $$
 \frac{\EE_{{p_0}}(\Gamma_t\vert \F_s)}{\EE_{p_0}(\Gamma_t)}
 =\frac{\1_{s<\tau}\PP_{X_s}(t-s<\tau)}{\PP_{p_0}(t<\tau)}
 =\frac{\1_{s<\tau}\PP_{X_s}(t-s<\tau)}{\int_E\PP_y(t<\tau){p_0}(\rd y)}.
 $$
 By Proposition \ref{prop:exist_eta} we have
 $$
 \PP_x(t<\tau)\underset{t\to\infty}{\sim}e^{-\lambda_0t}\eta(x)
 $$
 uniformly in $x\in \bar E$.
 As a consequence $\PP_{X_s}(t-s<\tau)\sim e^{-\lambda_0(t-s)}\eta(X_s)$ almost surely in the numerator, while by dominated convergence
 $ \int_E\PP_y(t<\tau){p_0}(\rd y)
 \sim
 \int_E e^{-\lambda_0 t}\eta(y){p_0}(\rd y)
 $
 in the denominator.
 As a consequence
 $$
 \frac{\EE_{{p_0}}(\Gamma_t\vert \F_s)}{\EE_{p_0}(\Gamma_t)}
 \xrightarrow[t\to\infty]{} M_s \eqqcolon
 \frac{\1_{s<\tau }e^{\lambda_0 s}\eta(X_s)}{\int_E \eta(y){p_0}(\rd y)}
 $$
 almost surely.
 By dominated convergence we also get that
 $$
 \EE_{p_0}(M_s)
 =\lim\limits_{t\to\infty}\EE_{p_0}\left(\frac{\EE_{{p_0}}(\Gamma_t\vert \F_s)}{\EE_{p_0}(\Gamma_t)}\right)
 = 1.
 $$
 By \cite[Theorem 2.1]{roynette2006some} this implies that $M_s$ is a martingale under $\PP_{p_0}$ and that
 $$
 \QQ^{t}_{q_0}(A_s)\xrightarrow[t\to\infty]{} \EE_{p_0}(M_s \1_{A_s})
 $$
 for any $A_s\in \F_s$.
 This implies that $\QQ_{q_0} \coloneqq \lim \QQ^t_{p_0}$ is well-defined, and moreover
 $$
 \QQ_{q_0}(A_s)=\EE_{p_0}(M_s \1_{A_s}) 
 =\int_{A_s} M_s(\omega)\, \PP_{p_0}(\rd\omega)
 \quad\Rightarrow\quad
 \left.\frac{\rd \QQ_{q_0}}{\rd\PP_{p_0}}\right|_{\F_s}=M_s.
 $$
 Notice moreover that
 $
 \1_{s<\tau}=\1_{X_s\in E}
 $.
 Recalling that $\eta$ vanishes on $\partial E$, we thus have that $ \1_{s<\tau}\eta(X_s)=\eta(X_s)$
 and therefore
 $$
 M_s=\frac{\1_{s<\tau }e^{\lambda_0 s}\eta(X_s)}{\int_E \eta(y){p_0}(\rd y)}
 =
 \frac{e^{\lambda_0 s}\eta(X_s)}{\int_E \eta(y){p_0}(\rd y)}
 $$
 as claimed.
 The explicit characterizations \eqref{eq:Pttilde} and \eqref{eq:def_generator_Ltilde} of $\tilde P_t$ and $\tilde \L$ immediately follow from from this expression for $\frac{\rd \QQ}{\rd \PP}$, see also \cite[Theorems 3.1 and 3.2]{CV}.
 The fact that the conditioned process is strong Markov (in fact, Feller continuous) is proved exactly as in \cite{CV}.

 In order to compute now the law $q_s$ of $X_s$ under the conditioned measure $ \QQ_{q_0}$, take $A_s=\{X_s\in B\}$ for any Borel set $B\in\mathcal B(E)$.
 Because $A_s\in\F_s$ we have from the previous computation
 $$
 q_s(B)= \QQ_{q_0}(X_s\in B)
 =\int_{A_s}M_s(\omega)\,\PP_{p_0}(\rd\omega).
 $$
 Substituting for $M_s$ from the previous equality readily gives
 $$
 q_s(B)
 =
 \frac{e^{\lambda_0 s}}{\int_E\eta(y){p_0}(\rd y)}\int_{X_s\in B}\eta(X_s(\omega))\PP_{p_0}(\rd\omega)
 =
 \frac{e^{\lambda_0 s}}{\int_E\eta(y){p_0}(\rd y)}\int_{ B}\eta(x)p_s(\rd x),
 $$
 which means exactly
 $$
 q_s(\rd x)=\frac{e^{\lambda_0 s}}{\int_E\eta(y){p_0}(\rd y)}\eta(x)p_s(\rd x).
 $$
 Our claim \eqref{eq:qs_from_ps} finally follows from the fact that $\eta$ is an eigenfunction, resulting in
 $$
 \frac{d}{ds}\int _{\bar E}\eta(x) p_s(\rd x)
 = \int_{\bar E}\eta(x) \Lt p_s(\rd x)
 = \int_{\bar E}\L\eta(x) p_s(\rd x)
 =-\lambda_0 \int_{\bar E}\eta(x) p_s(\rd x)
 $$
 and therefore $\int_E\eta(y){p_0}(\rd y)=  e^{\lambda_0 s}\int_E\eta(y)p_s(\rd y)$ in the previous denominator.
\end{proof}
In more tractable SDE terms, the conditioned process is characterized by
\begin{cor}
The probability measure
\begin{equation}
 \label{eq:def_pi}
 \pi(\rd x)\coloneqq \frac{1}{\int_E\eta(y)\alpha(\rd y)}\eta(x)\alpha(\rd x)
 \qquad \in \P(E)
\end{equation}
is the unique stationary distribution of the conditioned $Q$-process.
The conditioned generator is given by
\begin{equation}
\label{eq:Ltilde_explicit}
\tilde \L f(x) =\t\left[\Delta f(x)+\left(-\nabla U + 2\frac{\nabla\eta}{\eta}\right)\cdot \nabla f(x)\right],
\hspace{1cm}
x\in E,
\end{equation}
and the conditioned dynamics is given by
\begin{equation}
\label{eq:SDE_q}
\rd X_t =\sqrt{2\t(X_t)}\rd \tilde B_t +\t(X_t)\left(-\nabla U(X_t) + 2\frac{\nabla\eta(X_t)}{\eta(X_t)}\right)\rd t
\end{equation}
under $\QQ_{q_0}$ for any $q_0\in \P(E)$.
(Here $\tilde B_t$ is a standard Brownian motion.)
\end{cor}
The main point here is that, compared to \eqref{eq:SDE_Kimura}, the conditioning induces an additional logarithmic drift $2\nabla\log\eta=2\frac{\nabla\eta}{\eta}$ rendering the boundaries strongly repulsive.
Indeed, since $\eta$ vanishes linearly at both endpoints we have that $\frac{\nabla\eta}{\eta}$ blows-up  and eventually overtakes both the $-\nabla U$ drift and the Brownian fluctuation.
In particular particles never reach the boundary, as they should since the $Q$-process was precisely conditioned to never being absorbed, and therefore only takes values in $E$ (and not in $\bar E$).
More rigorously: from \eqref{eq:dPP/dQQ} it is not difficult to show that $\tau=+\infty$, $\QQ$-almost surely, and therefore no particular boundary conditions are needed in \eqref{eq:SDE_q} (although we will see later on that at the PDE level the corresponding Fokker-Planck equation should be interpreted with Neumann boundary conditions -- see Remark~\ref{rmk:Dir_Neum_BCs}, so the SDE can also be written more rigorously with reflection if really needed).
This also explains why, contrarily to \eqref{eq:def_L}, the new generator $\tilde \L$ only needs to be defined for $x\in E$ in \eqref{eq:Ltilde_explicit}, and not $x\in \bar E$.

\begin{proof}
 The fact that $\pi$ is stationary follows from \cite[Theorem 3.1]{CV}.
 Alternatively this can be checked by hand, simply choosing $p_0=\alpha$ in \eqref{eq:qs_from_ps}.
 Indeed in this case $q_0=\pi$, and since $\alpha $ is an eigenfunction we have $\mathring p_t=e^{-\lambda_0 t}\alpha$ whence $\eta p_t=\eta\mathring p_t=e^{-\lambda_0 t}\eta\alpha$ and
 $
 q_t=\frac{\eta e^{-\lambda_0 t}\alpha}{\int _E \eta e^{-\lambda_0 t}\alpha}
 =\frac{1}{\int_E \eta \alpha}\eta \alpha
 =\pi
 $.
 From \eqref{eq:def_generator_Ltilde} we compute
 \begin{multline*}
  \tilde \L f=\lambda_0 f+\frac 1\eta\L(\eta f)
  =\lambda_0 f + \frac 1\eta \t\left[\Delta(\eta f)-\nabla U\cdot \nabla(\eta f)\right]
  \\
  =\lambda_0 f + \frac 1\eta \t\left[\Big(\eta\Delta f + 2\nabla\eta\cdot \nabla f + f\Delta\eta\Big)-\nabla U\cdot \Big(\eta\nabla f + f \nabla\eta\Big)\right]\\
  =\lambda_0 f +\frac{f}{\eta}\underbrace{\t\left[\Delta \eta-\nabla U\cdot \nabla\eta\right]}_{\L\eta=-\lambda_0\eta}
   + \t\Big[\Delta f +\left(-\nabla U +2\frac{\nabla\eta}{\eta}\right)\cdot \nabla f\Big].
 \end{multline*}
This entails \eqref{eq:Ltilde_explicit} and the SDE \eqref{eq:SDE_q} immediately follows.
\end{proof}
A more variational and PDE-oriented characterization of the conditioning is as follows:
\begin{cor}
 Let
\begin{equation}
\label{eq:def_Utilde_V}
\tilde U\coloneqq U -2\log\eta
\qqtext{and}
V\coloneqq \tilde U+\log\t
\end{equation}
Then $\frac{e^{-\tilde U}}{\t}=e^{-V}\in  L^1(\rd x)$, the stationary measure $\pi=\eta\alpha$ can be written as the Gibbs distribution
\begin{equation}
\label{eq:pi_Gibbs}
\pi(\rd x)
=\frac{1}{\mathcal Z}\frac{e^{-\tilde U(x)}}{\t(x)}\rd x
=\frac{1}{\mathcal Z}e^{V(x)}\rd x
,
\end{equation}
and the adjoint generator reads
\begin{equation}
 \label{eq:Ltilde*_explicit}
\widetilde \L^* q
=\dive(\t \nabla q) +\dive(\t q\nabla V).
\end{equation}
\end{cor}
The previous repulsive effect of the conditioning should also be observed at this energetic level:
Since $\pi(\rd x)=\eta(x)\alpha(\rd x)$ vanishes linearly at the boundaries the new effective potential
$$
V=\tilde U + \log \t 
=cst-\log\pi =  \to +\infty
$$
blows-up logarithmically, while previously $\hat U=U+\log\t\to -\infty$ was blowing \emph{down} as anticipated in \eqref{eq:behaviour_all_potentials}.
Before conditioning, the boundaries were energetically attracting the dynamics towards $\hat U(\partial E)=-\infty$, while after conditioning, the $-2\log\eta$ term overtakes $\log \t$ and the boundaries become repulsive $\tilde U(\partial E)=+\infty$ (recall once again that both $\eta$ and $\t$ vanish linearly, while $U$ is at least bounded).
The law $q_s(\rd x)=Ce^{-\lambda_0 s} \eta(x) p_s(\rd x)$ accordingly does not charge $\partial E$.

\begin{proof}
Note first that $\eta,\t$ are smooth up to the boundary, where both vanish linearly.
As a consequence $\frac{e^{-\tilde U}}{\t}=\frac{e^{-U+\log\eta^2}}{\t}=\frac{\eta^2}{\t}e^{-U}$ is actually an honest continuous function, and as such belongs to $L^1(\rd x)$.

From \eqref{eq:Ltilde_explicit} the generator also reads
$$
\tilde \L f =\theta\left[\Delta f+\left(-\nabla U + 2\frac{\nabla\eta}{\eta}\right)\cdot \nabla f\right]
=\theta\Big[\Delta f+\nabla (-U + 2\log\eta)\cdot \nabla f\Big],
$$
whence the adjoint
$$
\tilde \Lt q =\Delta(\t q) +\dive\Big(q\t\nabla(\underbrace{U-2\log\eta}_{=\tilde U})\Big)
=\dive(\t \nabla q) +\dive\Big(q\t\nabla(\underbrace{\tilde U+\log\t}_{=V})\Big)
$$
as claimed in \eqref{eq:Ltilde*_explicit}.
In order to derive the Gibbs relation \eqref{eq:pi_Gibbs}, we use the standard identity $\nabla q=q\nabla\log q$  to rewrite $\tilde \Lt$ as
$$
\tilde \Lt q = \dive(\t \nabla q) +\dive(q\t\nabla V)
= \dive\left(\t q\nabla\log\left(\frac{q}{e^{-V}}\right)\right).
$$
It is now immediate to check that $Q=e^{-V}$ gives a stationary solution $\tilde \Lt Q=0$, and the result follows by uniqueness of the stationary distribution $\pi$.
\end{proof}

We finish this section with a PDE characterization of solutions to the abstract conditioned Fokker-Planck equation
\begin{equation}
 \label{eq:Fokker-Planck_q}
 \partial_t q_t =\tilde \Lt q_t.
\end{equation}
\begin{prop}
\label{prop:properties_sol_q_PDE}
Let $V$ be as in \eqref{eq:def_Utilde_V}.
The density $q_t(x)\in C^\infty(\R^+\times\bar E)$ of the $Q$-process in \eqref{eq:qs_from_ps} solves
\begin{equation}
 \label{eq:Fokker-Planck_q_PDE}
 \begin{cases}
  \partial_t q_t=\dive(\t \nabla q_t)+\dive(q_t\t \nabla V) \hspace{1cm}& t>0, x\in E\\
  q_t =0 & x\in \partial E
 \end{cases}
\end{equation}
in the classical sense.
Moreover $q_t\to q_0$ narrowly as $t\to 0^+$, and for all $t_0>0$ there exist $0<c_0<C_0$ depending on $t_0$ such that
\begin{equation}
 \label{eq:bounds_qt}
 c_0\pi(x)\leq q_t(x)\leq C_0\pi(x),
 \hspace{1cm}\forall\,t\geq t_0,x\in \bar E.
\end{equation}
\end{prop}

Note at this stage that we do not claim uniqueness of solutions to \eqref{eq:Fokker-Planck_q_PDE}.
This will be proved later in section~\ref{sec:relate_PDE_metric} in a very general sense, namely in the sense of $\EVI$ metric gradient flows.
The lower and upper bounds on $q_t(x)$ will be technically useful later on.

\begin{rmk}
 \label{rmk:Dir_Neum_BCs}
The Dirichlet boundary condition can (and should) actually be interpreted as a natural Neumann boundary condition $J_t\cdot\nu|_{\partial_E} =0$, where
$$
J_t\coloneqq\t \nabla q_t+q_t\t\nabla V
$$
is the effective flux $J_t$ in \eqref{eq:Fokker-Planck_q_PDE}.
This is consistent with our interpretation of the conditioned PDE as a Wasserstein gradient flow, which usually comes with no-flux boundary conditions so as to comply with the built-in mass conservation.
To see this, let us take for granted that $q_t(x)$ is smooth and vanishes linearly for $x\to\partial E$ as claimed.
Whence the first term $\t\nabla q_t$ vanishes linearly, while $\t q_t$ vanishes quadratically in the second factor.
Moreover, since $U$ is smooth and $\frac{\eta^2}{\t}$ vanishes linearly, we see that $\nabla V=\nabla U-\nabla\log\left(\frac{\eta^2}{\t}\right)$ behaves as $-\frac{1}{\t(x)}$ as $x\to\partial E$.
Therefore we get that $J_t\cdot \nu =0 $ is actually equivalent to $q_t\theta\frac{1}{\theta}\to 0$ on the boundaries, i-e
$$
J_t\cdot\nu=0
\quad\Leftrightarrow\quad
q_t=0,
\hspace{1cm}x\in \partial E. 
$$
\end{rmk}
\begin{proof}
 We first observe that, owing to our standing assumption \ref{ass:p0} that $\mathring p_0\not\equiv 0$, and due to $P_t\eta=e^{-\lambda_0t}\eta$ from the previous section, there holds
 $$
 \int_E\eta(x)\,p_t(\rd x)
 =\int_{\bar E}P_t\eta(x) \,p_0(\rd x)
 =e^{-\lambda_0t}\int_E\eta(x)\,p_0(\rd x)>0
 $$
 in the denominator of \eqref{eq:qs_from_ps}, with continuity as $t\to 0$.
 The $C^\infty(\R^+\times\bar E)$ smoothness and narrow continuity at $t=0^+$ of $q_t$ thus follow from those of $\eta(x)$ and $p_t$ (note that narrow continuity of $t\mapsto p_t$ is part of the requirement for very weak solutions in Definition~\ref{def:very_Weak_sol}, hence $\eta(x) p_t(\rd x)\rightharpoonup \eta(x) p_0(\rd x)$ in the numerator of \eqref{eq:qs_from_ps}).
 The PDE in \eqref{eq:Fokker-Planck_q_PDE} is given by our explicit computation \eqref{eq:Ltilde*_explicit} for $\tilde \Lt$.
 As for the Dirichlet condition, the continuity of $x\mapsto \mathring p_t(x)$ for fixed $t>0$ and the fact that $\eta(x)$ vanishes continuously on the boundary guarantees that $q_t(x)$ vanishes as well.
 
 Let us now turn to the bounds \eqref{eq:bounds_qt}.
 A naive argument would consist in trying to apply a comparison principle for \eqref{eq:Fokker-Planck_q_PDE}, using as sub and supersolutions $c_0\pi(x)$ and $C_0\pi(x)$ for small and large constants $c_0,C_0>0$, respectively (observe of course that $(\partial_t-\tilde\Lt)\pi=0$).
 However this requires
 \begin{enumerate}[(i)]
  \item 
  that $\partial_t-\tilde\Lt$ satisfy the weak maximum principle
  \item
  initially ordering $c_0\pi(x)\leq q_{t_0}(x)\leq C_0 \pi(x)$ at time $t_0>0$ for some well-chosen $c_0,C_0>0$
 \end{enumerate}
Unfortunately, due to the new logarithmic drift $2\frac{\nabla \eta}{\eta}$ the operator $\tilde\Lt$ has unbounded coefficients (in particular the zero-th order term), and is not even uniformly elliptic due to the lateral vanishing of $\t$.
Hence (i) is questionable to start with.
As for (ii), note that we actually allow $p_0\in \P(\bar E)$ and $q_0\in \P(E)$ to be singular in the interior: one expects that usual smoothing properties and the strong maximum principle would rather lead to $q_t(x)\geq c_0>0$ isntantaneously in time.
However this is incompatible with our statement, and clearly there is also more than meets the eye in (ii) due to the $\t$ degeneracy.
Instead we go back to the unconditioned level in order to retrieve information on $p$, which will be then translated to $q$.
 
 More precisely, recall from \eqref{eq:FP_with_BC} that the evolution for $\mathring p$ can be considered as a nonstandard, degenerate parabolic problem without boundary conditions
 $$
 \partial_t \mathring p =\Lt\mathring p=\Delta(\t\mathring p)+\dive(\t \mathring p\nabla U),
 $$
 but now with smooth coefficients.
 Results from \cite{feehan} guarantee that, for such degenerate operators, boundary conditions are only needed on the possibly nondegenerate portion of the boundary (where the diffusion does not vanish) in order to enable the (weak and strong) maximum principle.
 Here in our one-dimensional framework with $\theta(0)=\theta(1)=0$ this means that the comparison principle holds regardless of any ordering on the boundary, and only the initial ordering should be taken into account when applying the comparison principle (see also \cite{EM,epstein2013degenerate} for related topics).
 Since $\mathring p_0\geq 0$ is nontrivial, a first application of the strong maximum principle guarantees that $\mathring p_{t_0}(x)\geq c(t_0)>0$ for $t_0>0$ and any $x\in\bar E$.
 Recalling that the principal eigenfunction $\alpha_0(x)$ is smooth and bounded away from zero, we can therefore sandwich $c\alpha_0(x)\leq \mathring p_{t_0}(x)\leq C\alpha_0(x)$ upon choosing $c,C$ are small and large enough.
 Observing that $P(t,x)=e^{-\lambda_0t}\alpha_0(x)$ is a solution of $(\partial_t-\Lt)P=0$, the weak maximum principle therefore guarantees
 $$
 ce^{-\lambda_0 (t-t_0)}\alpha_0(x)\leq \mathring p_t(x)\leq  Ce^{-\lambda_0 (t-t_0)}\alpha_0(x),
 \hspace{1cm}t\geq t_0>0,\,x\in \bar E.
 $$
 Multiplying by $\eta(x)$ and dividing by $\int_E\eta(x)p_t(\rd x)=e^{-\lambda_0 t}\int_E\eta(x)p_0(\rd x)>0$ finally gives the desired result (recall that $\pi=\eta\alpha=C\eta\alpha_0$).
\end{proof}

From a purely analytical perspective, once $q_t$ has been determined by solving the conditioned Fokker-Planck equation \eqref{eq:Fokker-Planck_q_PDE}, one can reconstruct the full unconditioned distribution $p_t$ by first undoing the change of variables \eqref{eq:qs_from_ps}
as
$$
q_t(\rd x)=\frac{\eta(x)p_t(\rd x)}{\int_E \eta(y)p_s(\rd y)}
=\frac{\eta(x)\mathring p_t(\rd x)}{e^{-\lambda_0 t}\int_E \eta(y)\mathring p_0(\rd y)}
\quad\Rightarrow\quad
\mathring p_t(\rd x)=e^{-\lambda_0 t}\int_E \eta(y)\mathring p_0(\rd y)\frac{q_t(\rd x)}{\eta(x)},
$$
then leveraging the boundary kinematics \eqref{eq:kinematics_pboundary_from_pint} to compute $p^\partial_t$, and finally retrieve $p_t=p^\partial_t+\mathring p_t$.
This may appear slightly intriguing at first sight:
Indeed, the $Q$-process was obtained by conditioning to non-absorption, a $\PP$-negligible set of paths.
This construction therefore comes at the price of a total loss of probabilistic information, and it seems very surprising that one can reconstruct any statistics of the unconditioned dynamics from this negligible knowledge.
This actually goes even further: 
Not only the time marginal $p_t$ can be reconstructed from that of $q_t$, but in view of \eqref{eq:dPP/dQQ} even the path-measure $\PP$ can be inferred from $\QQ$, so in some sense it seems that the fullest possible unconditioned information can be retrieved.
As pointed out to us by N. Champagnat, this picture is however not completely correct:
The above reconstruction actually requires the explicit knowledge of the $\eta(x)$ eigenfunction, which comes from the unconditioned realm exclusively and cannot be determined from the conditioned generator $\tilde \L$ in any way.
In some sense, the eigenfunction $\eta$ really encodes much of the unconditioned information, at least sufficiently to bridge the two worlds, and there is no paradox.

%
%
\section{Metric setting}
\label{sec:metric}
Let us first motivate informally the metric theory of gradient flows \cite{AGS}.
Given a smooth function $F:\R^n\to\R$, the Cauchy-Schwarz inequality gives
$$
\frac{d}{dt} F(x_t)=\nabla F(x_t)\cdot \dot x_t\geq -\frac{1}{2}|\dot x_t|^2 -\frac{1}{2}|\nabla F(x_t)|^2
$$
for \emph{any} arbitrary $C^1$ curve $(x_t)_{t\geq 0}$.
Assume now that a particular curve $x_t$ satisfies the reverse inequality
\begin{equation}
\label{eq:EDI_formal}
\frac{d}{dt} F(x_t)\leq -\frac{1}{2}|\dot x_t|^2 -\frac{1}{2}|\nabla F(x_t)|^2.
\end{equation}
Then equality $\nabla F(x_t)\cdot \dot x_t=-\frac{1}{2}|\dot x_t|^2 -\frac{1}{2}|\nabla F(x_t)|^2$ holds in the Cauchy-Schwarz inequality, which implies
$$
\dot x_t=-\nabla F(x_t).
$$
In other words, the \emph{Energy Dissipation Inequality} \eqref{eq:EDI_formal} ($\EDI$ in short) is equivalent to the gradient flow.

Another useful characterization of the gradient-flow relies instead on convexity:
Assume now that the driving functional $F$ is in addition $\lambda$-convex, in the sense that the Hessian $D^2 F(x)\geq \lambda\operatorname{Id}$ for some $\lambda\in \R$.
By convexity the gradient $\nabla F(x)$ is completely characterized by
$$
\left\{F(y)\geq F(x) +\frac{\lambda}{2}|y-x|^2 + p\cdot(y-x) \mbox{ for all }y\right\}
\qqtext{if and only if}
\{p=\nabla F(x)\}.
$$
If now $x_t$ is an arbitrary $C^1$ curve we have $\frac{1}{2}\frac{d}{dt}|y-x_t|^2=\dot x_t\cdot(y-x_t)$, and imposing the \emph{Evolution Variational Inequality} ($\EVI_\lambda$ in short)
\begin{equation}
\label{eq:EVI_formal}
\frac{1}{2}\frac{d}{dt}|y-x_t|^2+\frac{\lambda}{2}|y-x_t|^2 + F(x_t)\leq F(y),
\hspace{1cm}\mbox{for all }y\in R^n\mbox{ and }t\geq 0
\end{equation}
is therefore also equivalent to $\dot x_t=-\nabla F(x_t)$.
Although convexity is not really essential for existence purposes, it is more suited for quantitative contraction and stability:
For, whenever $F$ is $\lambda$-convex and $x^1_t,x^2_t$ are two solutions, computing
$$
\frac{1}{2}\frac{d}{dt}|x^1_t-x^2_t|^2=(x^1_t-x^2_t)\cdot (\dot x^1_t-\dot x^2_t)
= -(x^1_t-x^2_t)\cdot (\nabla F(x^1_t)-\nabla F(x^2_t))\leq -\lambda |x^1_t-x_t^2|^2
$$
gives exponential contraction
$$
|x^1_t-x^2_t|\leq e^{-\lambda t}|x^1_0-x^2_0|.
$$
The fundamental observation is that both \eqref{eq:EDI_formal} and \eqref{eq:EVI_formal} have a clear counterpart in a purely metric setting, and can serve as a starting point for a theory of gradient flows in metric spaces.
This is developed in \cite{AGS} (see also \cite{santambrogio2017euclidean} for a gentle introduction), and we collect below (without proofs) the bare minimum material to serve our purpose later on.
\\

Let $(\X,\d)$ be a geodesic, complete metric space and $\E:\X\to(-\infty,+\infty]$ a function with proper domain $\Dom(\E)=\{x\in\X:\,\E(x)<+\infty\}$.
Mostly following \cite{AGS,muratori2020gradient}, we have
\begin{itemize}
 \item 
 An \emph{absolutely continuous} curve $\gamma:[0,T]\to\X$ belongs to $AC^p(0,T)$, $p\in [1,+\infty]$, if there exists $m\in L^p(0,T)$ such that
 $$
 \d(\gamma_{t_0},\gamma_{t_1})\leq \int_{t_0}^{t_1} m(t)\,\rd t,
 \hspace{1cm}\forall\,t_0,t_1\in [0,T].
 $$
 In this case $\gamma$ is absolutely continuous, the \emph{metric speed}
 \begin{equation}
 \label{eq:def_metric_speed_abstract}
 |\gamma'|(t)\coloneqq\lim\limits_{h\to 0}\frac{\d(\gamma_{t+h},\gamma_t)}{h}
 \end{equation}
 exists for a.e. $t\in (0,T)$, belongs to $L^p(0,T)$, and is the smallest function $m$ satisfying the above inequality.
 For $p=1$ we simply denote $AC=AC^1$, and the local-in time counterparts are denoted $AC^p_{\mathrm{loc}}$ with obvious definitions.
 \item
 We say that $g:\X\to[0,+\infty]$ is a \emph{strong upper gradient} for $\E$ if, for every absolutely continuous curve $\gamma\in AC(0,T)$, the function $g\circ\gamma$ is Borel and
 \begin{equation}
 \label{eq:chain_rule_metric}
 |\E(\gamma_{t_0})-\E(\gamma_{t_1})|\leq \int_{t_0}^{t_1} g(\gamma_t)|\gamma'|(t)\,\rd t,
 \hspace{1cm}\forall\,0<t_0\leq t_1<T.
  \end{equation}
In particular if $g\circ \gamma |\gamma'|\in L^1(0,T)$ then $\E\circ\gamma$ is absolutely continuous and
 $$
 \left|\frac{d}{dt}\E\circ\gamma(t)\right|\leq g(\gamma_t)|\gamma'|(t),
 \hspace{1cm}\mbox{for a.e. }t\in (0,T).
 $$

 \item
 The \emph{local slope} is
 \begin{equation}
 \label{eq:def_metric_slope_abstract}
 |\partial \E|(x)\coloneqq \limsup\limits_{y\to x}\frac{(\E(x)-\E(y))^+}{\d(x,y)}
 \end{equation}
 \item
 The function $\E$ is \emph{$\lambda$-geodesically convex} (or $\lambda$-convex, in short) if, for any $x,y\in \X$, there exists a geodesic $(\gamma_t)_{t\in[0,1]}$ joining $\gamma_0=x,\gamma_1=y$ along which
 $$
 \E(\gamma_t)\leq (1-t)\E(x)+t \E(y)-\lambda\frac{t(1-t)}{2}\d^2(x,y),
 \qquad 
 \forall t\in[0,1].
 $$
 \item
 Whenever $\E$ is $\lambda$-convex the local slope admits the global representation
 $$
 |\partial \E|(x)=\sup\limits_{y\neq x}\frac{(\E(x)-\E(y)+\lambda\d^2(x,y))^+}{\d(x,y)},
 $$
 has dense domain, is automatically lower semi-continuous, and is a strong upper gradient for $\E$.
 This relies on the fact that the global slope as above is always l.s.c. and a strong upper gradient, and $\lambda$-convexity guarantees that \emph{local=global}.
 We omit the subtle details and refer instead to \cite{AGS,muratori2020gradient}.
 \item
 A locally absolutely continuous map $\gamma\in AC^2_{\mathrm{loc}}([0,\infty))$ satisfies the \emph{Energy-Dissipation Inequality} $\EDI_{t_0}$ started from $t_0\geq 0$ if $\gamma_{t_0}\in \Dom(\E)$ and
 \begin{equation}
 \label{eq:EDI0}
 \int_{t_0}^t\left(\frac{1}{2}|\gamma'|^2(s)+\frac 12 |\partial\E|^2(\gamma_s)\right) \rd s + \E(\gamma_t)\leq \E(\gamma_0),\qquad\forall\,t>t_0.
 \tag{$EDI_{t_0}$}
 \end{equation}
 Whenever equality is achieved we say that $\gamma$ satisfies the \emph{Energy Dissipation Equality} $\EDE$.
 \item
 Fix $\lambda\in \R$.
 A solution of the \emph{Evolution Variational Inequality} $\EVI_\lambda$ is a continuous curve $\gamma:(0,\infty)\to \X$ such that $\E\circ\gamma\in L^1_{\mathrm{loc}}$ and, for every $y\in \Dom(\E)$,
 \begin{equation}
  \label{eq:def_abstract_EVI}
  \frac{1}{2}\frac{d}{dt}\d^2(\gamma_t,y)+\frac{\lambda}{2}\d^2(\gamma_t,y)+\E(\gamma_t)\leq \E(y)
  \tag{$\EVI_\lambda$}
 \end{equation}
 for a.e. $t>0$.
 
 An \emph{$\EVI_\lambda$-gradient flow} of $\E$ on a (possibly strict) subdomain $D\subseteq \overline{\Dom(\E)}$ is a family of continuous maps $(\mathsf S_t)_{t\geq 0}:D\to D$ such that, for every $\gamma_0\in D$,
 the strong semi-group property holds
 $$
 \mathsf S_{t+h}\gamma_0 =\mathsf S_h \mathsf S_t\gamma_0 \mbox{ for }t,h\geq 0
 \qqtext{with}
 \lim\limits_{t\to 0^+}\mathsf S_t \gamma_0=\mathsf S_0\gamma_0=\gamma_0,
 $$
 and
 $$
 t\mapsto \gamma_t=\mathsf S_t\gamma_0
 \qquad \mbox{is an }\EVI_\lambda \mbox{ solution}.
 $$
 \item
 Is is known \cite{AGS} that $\lambda$-convexity (almost) guarantees the well-posedness of an $\EVI_\lambda$ flow, but the converse is also true:
 If $\E$ generates an $\EVI_\lambda$-flow then necessarily $\E$ is $\lambda$-convex along \emph{all} geodesics, see \cite{daneri2008eulerian,monsaingeon2020dynamical}.
\end{itemize}
Fundamental properties of $\EVI_\lambda$ solutions are as follows
\begin{theo}[Properties of $\EVI_\lambda$-solutions, {\cite[Theorem 3.5,]{muratori2020gradient}}]
\label{theo:prop_EVI_sols}
\phantom{1}
 \par {\bf $\lambda$-contraction and uniqueness}
 \begin{equation*}
  \d(\gamma^1_t,\gamma^2_t)\leq e^{-\lambda t}\d(\gamma^1_0,\gamma^2_0),
  \hspace{1cm}\forall\,t\geq 0
 \end{equation*}
for any two solutions $\gamma^1,\gamma^2$.
\par
{\bf Regularizing effect}: $\gamma$ is locally Lipschitz, $\gamma_t\in \Dom(|\partial\E|)\subset \Dom(\E)$ for all $t>0$, and
$$
t\in[0,\infty)\mapsto \E(\gamma_t) \mbox{ is nonincreasing}
$$
\par
{\bf Functional inequalities and asymptotic behaviour}: if $\lambda>0$ and $\E$ has complete sublevels, then it admits a unique minimizer $\gamma_\infty\in\X$ satisfying
\begin{equation*}
 \frac{\lambda}{2}\d^2(x,\gamma_\infty)\leq \E(x)-\E(\gamma_\infty)\leq \frac{1}{2\lambda}|\partial\E|^2(x),
\end{equation*}
for all $x\in\X$ and
\begin{equation*}
 \E(\gamma_t)-\E(\gamma_\infty)\leq e^{-2\lambda(t-t_0)}\left(\E(\gamma_{t_0})-\E(\gamma_{t_\infty})\right)
\end{equation*}
\begin{equation*}
 \d(\gamma_t,\gamma_\infty)\leq e^{-\lambda(t-t_0)}\d(\gamma_{t_0},\gamma_\infty)
\end{equation*}
for all $t\geq t_0\geq 0$.
\end{theo}
Finally, a useful connection between notions of solutions of metric gradient flows will be for us
\begin{theo}[$\EDI\Rightarrow\EVI$]
\label{theo:EDIt0_EVI}
Assume that an $\EVI_\lambda$ gradient flow $\mathsf S_t$ of $\E$ exists on a (possibly strict) subdomain $D\subseteq\overline{\Dom(\E)}$.
Let $\gamma:[0,\infty)\to\E$ be an $\EDI_{t_0}$ solution for all $t_0>0$, whose image is contained in $D$.
If $\gamma$ is continuous at $t=0^+$ then $\gamma_t=\mathsf S_t\gamma_0$.
\end{theo}
Since it is well-known that $\EVI_\lambda\Rightarrow \EDE\Rightarrow\EDI$, this means that all three notions are in fact equivalent, provided that the solution under consideration remains in the domain of the flow $\mathsf S$.
\begin{proof}
This exact statement can be found in \cite[theorem 4.2]{muratori2020gradient}, assuming that $\EDI_0$ holds up to $t=0$ and not up to any $t_0>0$ (note however that the continuity of $\gamma$ at $t=0^+$ is contained in the definition of $\EDI_0$).
Observe that \eqref{eq:def_abstract_EVI} is in fact local in time:
applying \cite[theorem 4.2]{muratori2020gradient} thus shows that $\gamma$ is an $\EVI_\lambda$ solution in the interval $[t_0,+\infty)$ for any $t_0>0$, thus in $(0,+\infty)$.
The continuity $\gamma_t\to\gamma_0$ imposed at $t\to 0^+$ allows to conclude.
\end{proof}

\subsection{The Wasserstein-Shahshahani distance}
Let us recall that our whole analysis relies upon choosing an ad-hoc geometry, that is optimal transport over the Shahshahani space.
More precisely, recall that we choose to view $E=(0,1)$ as a Riemannian manifold with
metrics induced by $\t$
 $$
 \langle\xi_1,\xi_2\rangle_\t \coloneqq \frac{\langle\xi_1,\xi_2\rangle}{\t(x)}
 ,\qquad\forall\, \xi_1,\xi_2\in T_x,
 \,x\in E.
 $$
 We denote by
 $$
 |\xi|^2_{\t}\coloneqq \frac{|\xi|^2}{\t(x)}
 $$
 the corresponding norm on $T_x$.
 For the specific choice $\t(x)=x(1-x)$, which is our guiding thread in this work, this is known as the Shahshahani metrics and was introduced in \cite{shahshahani1979new} to study a Replicator Equation from evolutionary genetics as a gradient system.
 Given the biological background of our Kimura model, it is no surprise that this Shahshahani structure should appear again here and we stick to the terminology.
 Although singular at $x=0,1$, this induces a proper distance on $\bar E$ \cite[Lemma 10]{chalub2021gradient}
 \begin{equation}
 \label{eq:explicit_d_Shahshahani}
 \dS(x,y)
 \coloneqq 
 \inf\limits_{\substack{
u_0=x\\
 u_1=y  
}
 } \int_0^1 |\dot u_t|^2_{\t}\,\rd t
 =\int_x^y\frac{1}{\sqrt{\t(z)}}\,\rd z.
 \end{equation}
 Note that the linear behaviour of $\t$ at the boundaries makes this quantity finite even for $x=0$ or $y=1$, and that
 \begin{equation}
 \label{eq:dS_behaves_sqrt}
\dS(x,\partial E)\sim C \sqrt{\d_\R(x,\partial E)} \qquad \mbox{as } x\to\partial E.
 \end{equation}
 The space $(\bar E,\dS)$ has finite diameter $D=\int_0^1\frac{\rd z}{\sqrt{\t(z)}}$, and it is not difficult to check that it is complete.
 The corresponding \emph{Wasserstein-Shahshahani distance} over $\P(\bar E)$ is then simply defined as the usual quadratic Wasserstein distance  \cite{villani_BIG} over the Polish space $(\bar E,\dS)$,
 $$
 \W_\t^2( q , r )\coloneqq \min\limits_{\gamma\in \Gamma(q,r)}\iint_{\bar E^2}\dS^2(x,y)\rd\gamma(x,y)
 $$
 for any probability measures $q,r\in \P(\bar E)$.
 Here the minimization runs over admissible plans
 $$
 \gamma\in\Gamma( q , r )=\left\{\pi(\rd x,\rd y)\in \P(\bar E\times\bar E) \mbox{ with marginals }\pi_x= q ,\pi_y= r \right\}.
 $$
General Optimal Transport theory \cite{villani_BIG,villani2003topics} gives that $(\P(\bar E),\W_\t)$ turns into a Polish space of its own.
Note also that, since $\dS$ is locally equivalent to the Euclidean metric and $1/2$-H\"older at the boundaries, we have that $|x_n-x|_\R\to 0$ if and only if $\dS(x_n,x)\to 0$ and therefore the set of real-valued $\dS$-continuous functions over $\bar E$ coincides with the usual set $C(\bar E)$ of continuous functions up to the boundary.
In particular $\W_\t$ is well-known to metrize the narrow convergence of measures \cite{villani_BIG}, i.e.
 \begin{center}
$\int_{\bar E}\varphi(x) q _n(\rd x)\to \int_{\bar E}\varphi(x) q (\rd x)$ for all $\varphi\in C(\bar E)$
\qquad 
if and only if
\qquad 
$\W_\t( q _n, q )\to 0$.
\end{center}
In order to emphasize the dependence on the Riemannian structure and differentiate from the Euclidean framework we use $\t$-subscripts in the sequel for any quantity computed with respect or related to this geometry.
For example we write $\P_{\t}=\P(\bar E)$ equipped with $\W_\t$, the metric speed of a curve of measures $ q _t$ computed w.r.t $\W_\t$ is denoted $| q '|_{\W_\t}(t)$, we write $AC_\t$ for $\t$-absolutely continuous curves of measures, and so on (with sometimes a slight abuse of notations).
By contrast and whenever needed, we use the $\R$ subscript to denote any quantity measured using the Euclidean norm $|.|_\R$.
\subsection{Absolutely continuous curves}
A fundamental result in Optimal Transport is the Benamou-Brenier formula \cite{BB}, yielding a dynamical interpretation of the static Wasserstein distance.
This allows the representation of $AC^2$ curves in the Wasserstein space as solutions of continuity equations
\begin{equation}
\label{eq:CE}
\partial_t q _t+\dive( q _t v_t)=0,
\tag{CE}
\end{equation}
which in turn reveals crucial in identifying solutions of abstract metric gradient-flows in the Wasserstein space as actual PDEs.
However, since our pseudo-Riemannian manifold is singular on the boundaries, we cannot simply apply the existing theory.
We opted instead for a self-contained presentation, mostly relying on \cite{lisini2007characterization} and trying to keep the technical details to a bare minimum (many of the arguments below are mere adaptations of known results, fitted to our purposes).

\begin{defi}
\label{def:weak_sols_CE}
 We say that a curve of measures $( q _t)_{t\in I}$ on $ \P(\bar E)$ and a (Borel) velocity field $v_t(x)$ solve the continuity equation \eqref{eq:CE} in $(t,x)\in I\times \bar E$ if
 $$
 \frac{d}{dt}\int_{\bar E}\varphi(x)\,  q _t(\rd x)=\int_{\bar E}\nabla\varphi(x)\cdot v_t(x) \, q _t(\rd x)
 $$
 holds in the sense of distributions in $I$ for all $\varphi\in C^1(\bar E)$.
\end{defi}
Note that this encodes a no-flux boundary condition $ q  v\cdot  \nu =0$ on $\partial E$ in a weak sense.
We denote the weighted $L^2$ norm (kinetic energy) as
$$
\|v_t\|^2_{L^2_\t( q _t)}
\coloneqq
\int_{\bar E}|v_t(x)|_\t^2\, q _t(\rd x)
=
\int_{\bar E}\frac{|v_t(x)|^2}{\t(x)}\, q _t(\rd x)
$$
whenever this quantity is finite.
Note that the very meaning of this integral is far from being clear:
Since $\frac{1}{\t(x)}\approx \frac{1}{x(1-x)}$ is not continuous up to the boundary, it may a priori not be legitimate to integrate w.r.t to an arbitrary measure $q\in \P(\bar E)$, in particular if $q$ has atoms on the boundary.
This can be made rigorous upon understanding the quotient $\frac{|\xi|^2}{\t(x)}$ in a particular sense, which will be done very soon in Lemma~\ref{lem:udot} -- see in particular \eqref{eq:def_Lagrangian}.
Leaving this issue aside, the Benamou-Brenier formula is expected to read at least formally
$$
\W^2_\t( q _0, q _1)=\min\Bigg\{\int_0^1\|v_t\|^2_{L^2_\t( q _t)}\rd t
\mbox{ s.t. }( q ,v)\mbox{ solve \eqref{eq:CE} with } q |_{t=0}= q _0, q |_{t=1}= q _1
\Bigg\},
$$
and suggests that $AC^2(I;\P_\t)$ curves should be represented by suitable velocity fields solving \eqref{eq:CE} with finite weighted kinetic action.
This is indeed the case: 
\begin{theo}
\label{theo:AC2_curves}
 For any curve $ q \in AC^2(I;\P_\t)$ there exists a minimal Borel vector-field $\tilde v_t(x):I\times \bar E\to\R$ such that $( q ,\tilde v)$ solves the continuity equation and
 \begin{equation}
   \|\tilde v_t\|_{L^2_\t( q _t)}  \leq  | q '|_{\W_\t}(t)
 \hspace{1cm}\mbox{for a.e. }t\in I.
  \label{eq:tilde_v_leq_mu'}
  \end{equation}
 Conversely, any solution $( q ,v)$ of the continuity equation with finite kinetic energy ${\int_I\|v_t\|^2_{L^2_\t( q _t)}\rd t<\infty}$ gives an absolutely continuous curve $ q \in AC^2(I;\P_\t)$ with
 \begin{equation}
 | q '|_{\W_\t}(t)\leq \|v_t\|_{L^2_\t( q _t)}.
 \hspace{1cm}\mbox{for a.e. }t\in I.
 \label{eq:mu'_leq_v}
 \end{equation}
 In particular, for any given $ q \in AC^2(I;\P_\t)$ the minimal vector-field $\tilde v$ as above is unique, and such that $| q '|_{\W_\t}(t)=\|\tilde v_t\|_{L^2_\t( q _t)}$ for a.e. $t$.
\end{theo}
The proof below will rely on the following elementary lemma:
\begin{lem}
 Define the Lagrangian
 \begin{equation}
 \label{eq:def_Lagrangian}
 L(x,\xi)=\begin{cases}
         \frac{|\xi|_\R^2}{\t(x)} & \mbox{if }x\in E,\\
         0 & \mbox{if }x\in\partial E\mbox{ and }\xi=0,\\
         +\infty & \mbox{if }x\in\partial E \mbox{ but }\xi\neq 0.
        \end{cases}
 \end{equation}
 For any curve $u\in AC^2(I;\bar E_\t)$ and a.e. $t\in I$, the classical derivative
 $$
 \dot u_t=\lim\limits_{ h\to 0}\frac{u_{t+h}-u_t}{h}
 $$
 exists and
 \begin{equation}
 \label{eq:udot_Luxi}
 |u'|_\t^2(t)=L(u_t,\dot u_t).
 \end{equation}
 \label{lem:udot}
\end{lem}
Let us emphasize once again that, in the left-hand side of \eqref{eq:udot_Luxi}, $|u'|_\t$ denotes the metric speed computed w.r.t. the Shahshahani distance $\dS$.
\begin{proof}
 Observe first that $u\in AC^2(I;\bar E_\t)$ is $\dS$-continuous in time, hence also continuous in the classical sense (the $\dS$ topology being equivalent to the Euclidean topology).
As a consequence $u_{t_0}$ can be unambiguously evaluated at any $t_0\in I$, and we treat separately $u_{t_0}\in\partial E$ and $u_{t_0}\in E$.
 The case of an interior point $u_{t_0}\in E$ falls within standard, smooth Riemannian geometry and the proof is immediate, see e.g. \cite[Section 2.1]{lisini2009nonlinear}.
 For a boundary point, assume e.g. $u_{t_0}=0$ (the case $u_{t_0}=1$ is similar), and assume without loss of generality that $t_0$ is a ``metric differentiability point'' for $u$ (i-e the metric derivative $|u'|_\t(t_0)$ exists, which is indeed guaranteed for a.e. $t$ by standard properties of the metric speed).
 Then by \eqref{eq:dS_behaves_sqrt} we have that
 $$
 \frac{\sqrt{u_{t_0+h}}-\sqrt{u_{t_0}}}{h}
 =\frac{\sqrt{u_{t_0+h}}}{h}
 \sim C\frac{\dS(u_{t_0+h},u_{t_0})}{h}\to C|u'|_\t(t_0)<+\infty,
 $$
 i-e $t\mapsto\sqrt{u_t}$ is differentiable at $t_0$.
 By standard chain-rule we see that the usual derivative
 $$
 \dot u_{t_0}=2\sqrt{u_{t_0}}\left.\frac{d}{dt}\right|_{t_0}\sqrt{u_t}=0
 $$
 exists and vanishes, hence by definition $L(u_{t_0},\dot u_{t_0})=0$.
 The subtle issue here is that, although $L(u_{t_0},\dot u_{t_0})=0$ at that particular time $t=t_0$, it might happen that $|u'|_\t(t_0)\neq 0$ there, thus violating our claim that $|u'|_\t^2(t)=L(u_t,\dot u_t)$.
 Note however that we only claim equality for almost every $t$: 
 Hence, it suffices to prove that the ``bad set'' $N\subset I$ of such times $t_0\in I$ where $u_{t_0}=0$ (or $u_{t_0}=1$) with $|u'|_\t(t_0)\neq 0$ is negligible.
 By definition, having non-zero metric speed implies that the curve is really moving, and more precisely $\dS(u_t,u_{t_0})\sim |u'|_\t(t_0)|t-t_0|\neq 0$ for $t\neq t_0$ close enough to $t_0$.
 As a result $u_t\neq 0$ for such times, which means that $N$ is actually a discrete set and therefore Lebesgue-negligible as claimed.
  \end{proof}
\begin{proof}[Proof of Theorem~\ref{theo:AC2_curves}]
 With Lemma~\ref{lem:udot} now at hand, the proof is almost identical to \cite[Theorem 2.4]{lisini2007characterization} so we briefly give the main lines and only highlight the slight differences.
 The main idea is that, in an abstract metric framework, absolutely continuous curves in an overlying Wasserstein space can always be recovered as suitable superposition of absolutely continuous curves in the associated underlying space \cite{lisini2007characterization}.
 Whenever that underlying metric space is in fact a Riemannian manifold with induced Riemannian distance, one can go one step further and relate the metric speed of an underlying curve to the Riemannian norm of its usual velocity (its time derivative, i-e a tangent vector).
 This allows reconstructing the velocity field $\tilde v_t(x)$ in Theorem~\ref{theo:AC2_curves} essentially by superposing  the velocities of all Lagrangian particles simultaneously sitting at position $x$ at time $t$.
 The only subtle point here, compared to \cite{lisini2007characterization}, is that our underlying metric space $(\bar E,\dS)$ is not properly speaking a smooth Riemannian manifold and therefore the singular boundary points must be dealt with separately, which is done thanks to Lemma~\ref{lem:udot}.
 
 More precisely, let $\Gamma=C(I;\bar E_{\t})$ be the space of continuous curves in $\bar E$ with induced $\dS-\sup$ distance.
 Owing to \cite[Theorem 5]{lisini2007characterization}, valid in general metric spaces, and given a curve $ q \in AC^2(I;\P_\t)$ as in our statement, there is a path-measure $\eta\in\P(\Gamma)$ such
 \begin{enumerate}[(i)]
  \item 
  $\eta$ is concentrated on $AC^2(I;\bar E_\t)$
  \item
  $ q _t={e_t}_\#\eta$ for all $t\in I$, where $e_t(u)=u_t$ is the time-$t$ evaluation map
\item  
the metric speed is represented by superposition
$$
| q '|^2_{\W_{\t}}(t)=\int_\Gamma|u'|^2_{\t}(t)\,\eta(\rd u)
$$
 \end{enumerate}
By disintegration, there exists a (unique) Borel family of probability measures $\eta_{t,x}\in \P(\Gamma)$ and concentrated on $\Gamma_{t,x}=\{u\in\Gamma:\quad u_t=x\}$ such that
\begin{equation*}
\int_{I\times \Gamma}\varphi(t,u)\,\rd t\, \eta(\rd u)
=
\int_{I\times\bar E}\left(\int_\Gamma\varphi(t,u)\,\eta_{t,x}(\rd u)\right) q _t(\rd x)\rd t
\end{equation*}
for all $\varphi\in C(I\times \Gamma)$.
The velocity field $\tilde v$ is then defined as the superposition of all Lagrangian speeds
\begin{equation}
\label{eq:def_v_tilde}
\tilde v_t(x)\coloneqq \int_\Gamma\dot u_t\,\eta_{t,x}(\rd u).
\end{equation}
A keypoint is that, exactly as in \cite{lisini2009nonlinear}, this is unambiguously defined for $\rd q _t\otimes\rd t$-almost every $x,t$ because the paths $u$ charged by $\eta$ (or $\eta_{t,x}$) are $AC^2$, here specifically $AC^2(I;\bar E_\t)$.
Hence the classical derivative $\dot u_t$ is well-defined for a.e. $t\in I$ according to Lemma~\ref{lem:udot}.
By Jensen's inequality and our representation of the metric speed $|u'|_\t^2(t)=\frac{|\dot u_t|_\R^2}{\t(u_t)}$ from Lemma~\ref{lem:udot} we get
\begin{multline*}
 \int_I\|\tilde v_t\|^2_{L^2_\t( q _t)}\rd t
 =\int_{I\times\bar E}\frac{|\tilde v_t(x)|^2}{\t(x)} q _t(\rd x)\rd t
 \\
 =\int_{I\times\bar E}\frac{\left|\int_\Gamma \dot u_t\,\eta_{t,x}(\rd u)\right|^2}{\t(x)} q _t(\rd x)\rd t
 \leq
 \int_{I\times\bar E}\left(\int_\Gamma \frac{|\dot u_t|^2}{\t(x)}\eta_{t,x}(\rd u)\right) q _t(\rd x)\rd t
 \\
 =
 \int_{I\times \Gamma}\frac{|\dot u_t|^2}{\t(u_t)}\,\rd t\,\eta(\rd u)
 \overset{\eqref{eq:udot_Luxi}}{=}\int_{I\times \Gamma}|u'|^2_{\t}(t)\rd t\,\eta(\rd u)
 =\int_I| q '|^2_{\W_\t}(t)\,\rd t.
\end{multline*}
One can actually carry over the exact same computation in any time interval $[t_0,t_1]\subset I$,
which allows to conclude that $\int_{t_0}^{t_1}\|\tilde v_t\|^2_{L^2_\t( q _t)}\rd t\leq \int_{t_0}^{t_1}| q '|^2_{\W_\t}(t)\rd t$ and thus
$$
\|\tilde v_t\|_{L^2_\t(\rd q _t)} \leq | q '|(t)\qquad \mbox{for a.e. }t\in I.
$$
In order to prove that $( q ,\tilde v)$ solve the continuity equation as claimed, observe first that, from \eqref{eq:explicit_d_Shahshahani} and the mean value Theorem, we have $ \dist(x,y)=\int_x^y\frac{1}{\sqrt{\t(z)}}\rd z=\frac{|x-y|}{\sqrt{\t(z)}}$ for some $z=z_{x,y}\in(x,y)$ and as a consequence
$$
 |x-y|\leq \|\sqrt\t\|_\infty\dist(x,y)
$$
for any $x,y\in\bar E$.
In particular any $\d_\R$-Lipschitz map $\varphi:\bar E\to\R$ is $\dS$-Lipschitz as well.
Take now any $\varphi\in C^1(\bar E)$ and write, for any optimal plan $\gamma_{t_0,t_1}$ from $ q _{t_0}$ to $ q _{t_1}$,
\begin{multline*}
\left|\int_{\bar E} \varphi\, q _{t_1}- \int_{\bar E} \varphi\, q _{t_0}\right|
=\left|\int_{{\bar E}\times{\bar E}} \left(\varphi(x)-\varphi(y)\right) \gamma_{t_0,t_1}(\rd x,\rd y)\right|
\leq \int_{{\bar E}\times{\bar E}}|\varphi(x)-\varphi(y)|\gamma_{t_0,t_1}(\rd x,\rd y)
\\
\leq L \int_{{\bar E}\times{\bar E}}\dS(x,y)\gamma_{t_0,t_1}(\rd x,\rd y)
\leq L \left( \int_{{\bar E}\times{\bar E}}\dS^2(x,y)\gamma_{t_0,t_1}(\rd x,\rd y)\right)^{\frac{1}{2}}
= L\,\W_{\t}( q _{t_0}, q _{t_1}),
\end{multline*}
where of course $L=\operatorname{Lip}_{\dS}(\varphi)$.
Since $ q \in AC^2([0,1];\P_\t)$ this shows that $t\mapsto\int_{\bar E}\varphi(x) q _t(\rd x)$ is $AC^2$ (in the classical sense), hence we can legitimately differentiate under the integral sign as
\begin{multline*}
 \frac d{dt}\int_{\bar E}\varphi(x)\, q _t(\rd x)
 =\frac d{dt}\int_\Gamma\varphi(u_t)\,\eta(\rd u)
 =\int_\Gamma \left[\frac d{dt}\varphi(u_t)\right]\,\eta(\rd u)
 \\
 =\int_\Gamma\nabla\varphi(u_t)\cdot \dot u_t\,\eta(\rd u)
 =\int_0^1\int_{\bar E} \nabla\varphi(x)\cdot \tilde v_t(x)\, q _t(\rd x)\rd t.
\end{multline*}
Here the last equality follows from the very definition \eqref{eq:def_v_tilde} of $\tilde v$ as $\eta$-superposition of Lagrangian velocities $\dot u$.

Let us now turn to the converse implication, and assume that $( q ,v)$ solve the continuity equation in $I\times {\bar E}$ with finite $L^2_\t$ kinetic energy.
Without loss of generality we assume that $I=(0,T)$.
The key step is to apply a probabilistic representation for solutions of the continuity equation, allowing somehow to retrieve a notion of Lagrangian trajectories although the driving velocity field $v_t(x)$ could be very rough, see \cite[Section 8.2]{AGS} and \cite{ambrosio2008transport}.
Observe that $\frac{|v_t(x)|^2}{\|\t\|_\infty} \leq \frac{|v_t(x)|^2}{\t(x)} $, hence $( q ,v)$ solve the continuity equation with finite Euclidean kinetic energy $\int_I\int_{\bar E} |v_t(x)|_\R^2\, q _t(\rd x)\,\rd t<\infty$.
This allows a direct application of \cite[Theorem 8.2.1.i]{AGS} and we conclude that there exists a probability measure $\eta\in P({\bar E}\times\Gamma)$ (here $\Gamma=C(I;{\bar E})$), supported on pairs $(x,u)$ such that $u\in AC^2(I;\bar E_\R)$  with $u_0=x$ and solving $\dot u_t=v_t(u_t)$ in the integral sense.
Moreover for all $t\in I$ the measure $ q _t$ is the pushforward of $\eta$ by the evaluation map $e_t$, i-e
$$
\int_{\bar E} \varphi(x) q _t(\rd x)=\int_{{\bar E}\times\Gamma}\varphi(u_t)\,\eta(\rd x,\rd u)
$$
for all $\varphi\in C({\bar E})$.
The second marginal $\zeta\in\P(\Gamma)$ of $\eta$ is thus supported on $AC^2_\R$ solutions of $\dot u_t=v_t(u_t)$ and such that
\begin{equation}
\int_{\bar E} \varphi(x) q _t(\rd x)=\int_{\Gamma}\varphi(u_t)\,\zeta(\rd u).
\label{eq:def_zeta_path_measure}
\end{equation}
Now for any $t_0<t_1$ the measure $\gamma_{t_0,t_1}\coloneqq (e_{t_0},e_{t_1})_\#\zeta\in \P({\bar E}\times{\bar E})$ is an admissible coupling between $ q _{t_0}, q _{t_1}$ and therefore
\begin{equation*}
 \W_\t^2( q _{t_0}, q _{t_1})\leq \int_{{\bar E}\times{\bar E}}\dS^2(x,y)\rd \gamma_{t_0,t_1}(x,y)
 =\int_{\Gamma}\dS^2(u_{t_0},u_{t_1})\,\zeta(\rd u).
\end{equation*}
Since $\dot u_t=v_t(u_t)$ for $\zeta$-a.e. $u\in\Gamma$, we have, by definition \eqref{eq:explicit_d_Shahshahani} of $\dS$, interpreting the quotient $\frac{|\xi|^2}{\t(x)}$ as before, and scaling time appropriately,
\begin{multline*}
\int_{\Gamma}\dS^2(u_{t_0},u(t_1))\,\zeta(\rd u)
\leq 
\int_{\Gamma}\left((t_1-t_0)\int_{t_0}^{t_1}\frac{|\dot u_t|^2}{\t(u_t)}\,\rd t\right)\zeta(\rd u)
\\
= (t_1-t_0)  \int_{\Gamma}\int_{t_0}^{t_1}\frac{|v_t(u_t)|^2}{\t(u_t)}\,\rd t\, \zeta(\rd u)
=(t_1-t_0)  \int_{t_0}^{t_1}\int_{\Gamma}\frac{|v_t(u_t)|^2}{\t(u_t)}\, \zeta(\rd u)\,\rd t
\\
=(t_1-t_0)  \int_{t_0}^{t_1}\int_{{\bar E}}\frac{|v_t(x)|^2}{\t(x)}\,  q _t(\rd x)\,\rd t.
\end{multline*}
Here we made a crucial use of \eqref{eq:def_zeta_path_measure} for the last equality.
Since $t_0<t_1$ were arbitrary this shows that $ q \in AC^2(I;\P_\t)$ and entails \eqref{eq:mu'_leq_v}.

Finally, uniqueness of the minimal velocity field $\tilde v$ is an easy consequence of the linearity of the continuity equation constraint and of the strong convexity of the $L^2$ norm.
\end{proof}

\subsection{Displacement convexity and metric slope}
\label{subsec:convexity_slope}
Let us recall from Section~\ref{sec:Kimura_Qprocess} the expressions
$$
\pi(\rd x)=\frac{1}{\int_{\bar E}\eta(x)\alpha(\rd x)}\eta(x)\alpha(\rd x)
=\frac{1}{\mathcal Z}e^{-(\tilde U+\log\t)(x)}\rd x =
\frac{1}{\mathcal Z}e^{-V(x)}\rd x
$$
for our stationary measure of the $Q$-process.
For reasons that shall become clear later on, let us denote
\begin{equation}
 \label{eq:def_V_W}
 W\coloneqq V-\log\sqrt\t.
\end{equation}
Typically writing
$$
\sigma(x)=\frac{\rd  q }{\rd \pi}(x)
$$
for the Radon-Nykodim derivative w.r.t. the reference measure, the relative entropy $\Hpi:\P(\bar E)\to (-\infty,\infty]$ is classically defined as
\begin{equation}
\label{eq:def_Hpi}
 \Hpi( q )\coloneqq 
 \begin{cases}
  \int_{\bar E} \sigma(x)\log\sigma(x) \,\pi(\rd x) & \mbox{if } q (\rd x)=\sigma(x)\pi(\rd x),\\
  +\infty & \mbox{else}.
\end{cases}
\end{equation}
This is well-known to be lower semi-continuous w.r.t. narrow convergence of measures.
Moreover it is not difficult to check that $\H_\pi$ has dense domain $\overline{\Dom(\H_\pi)}=\P_\theta(\bar E)$.
%
Note that, since $\pi(\rd x)=C\eta(x)\alpha_0(x)\rd x$ is absolutely continuous, any $q\ll\pi$ is also absolutely continuous w.r.t. to the Lebesgue measure.
Thus writing $q(\rd x)=q(x) \rd x$, our entropy \eqref{eq:def_Hpi} can also be written as
$$
\H_\pi(q)=\int_{\bar E}q(x)\log q(x)\,\rd x + \int_{\bar E} q(x)V(x)\,\rd x.
$$
(Here we assumed w.l.o.g. that the normalizing factor $\mathcal Z=1$ and $V=-\log \pi$.)

\begin{prop}
\label{prop:geodesic_convexity_alpha}
Let $W$ be as in \eqref{eq:def_V_W}.
Then
\begin{equation}
 \label{eq:def_lambda_mod_convexity}
 \lambda\coloneqq 
 \min\limits_{x\in(0,1)}\left\{\sqrt{\t}\frac{d}{dx}\left(\sqrt{\t}\frac{d}{dx}W\right)\right\}.
 \end{equation}
 is a finite real number and the entropy $ q \mapsto \H_\pi( q )$ is $\lambda$-displacement convex along any $\W_\t$-geodesic $( q _t)_{t\in[0,1]}$ between $ q _0, q _1\in \P_\t(E)$,
 \begin{equation}
 \label{eq:convexity_E}
 \H_\pi( q _t)\leq (1-t)\H_\pi( q _0)+t\H_\pi( q _1)-\lambda\frac{t(1-t)}{2}\W_\t^2( q _0, q _1),
 \qquad \forall\,t\in[0,1].
 \end{equation}
\end{prop}

\begin{proof}[Direct but formal proof]
It is by now well understood \cite[Theorem 1.3]{sturm2005convex} that, on a smooth, complete Riemannian manifold $(M,g)$, and given a smooth potential $W:M\to \R$, the functional
\begin{equation}
\mathcal S_W( q )= \int_M \rho(x)\log\rho(x)\, \Vol (\rd x) + \int_M W(x)\,q(\rd x) ,
\qqtext{with}
q(\rd x)=\rho(x)\Vol(\rd x)
\label{eq:def_entropy_Sturm}
\end{equation}
is $\lambda$-geodesically convex in the Wasserstein space $\mathcal P_g(M)$ (built upon the Riemannian distance induced by $g$) if and only if
\begin{equation}
\Ric_x + \Hess_x W\geq \lambda,
\hspace{1cm}\forall\,x\in M.
\end{equation}
Note now that the standard formula $\Vol(\rd x)=\sqrt{\operatorname{Det}g_x}\rd x$ gives here $\Vol(\rd x)=\frac{1}{\sqrt{\t(x)}}\rd x$.
Changing reference measure $q(\rd x)= \sigma(x)\pi(\rd x)=\rho(x)\Vol(\rd x)$, straightforward algebra shows that our entropy reads in the more intrinsically Riemannian form
$$
\H_\pi(q)=\mathcal S_W(q)
\qqtext{with} W \mbox{ as in }\eqref{eq:def_V_W}.
$$
In our flat one-dimensional geometry we have of course $\Ric_x\equiv 0$ hence it suffices to check $\Hess W\geq \lambda$.
In intrinsic Riemannian coordinates $y$ we have $\partial_y =\frac1{\sqrt g_x}\partial_x=\sqrt{\t(x)}\partial_x$, hence \eqref{eq:def_lambda_mod_convexity} is exactly $\lambda\coloneqq \min \Hess W$ expressed in Euclidean coordinates.
\end{proof}
This proof is of course not acceptable because $(\bar E,\dS)$ is not a smooth, complete Riemannian manifold.
We shall bypasses this issue thanks to a purely one-dimensional isometry argument.
\begin{proof}
To begin with, let us briefly justify that \eqref{eq:def_lambda_mod_convexity} is indeed finite.
Rewinding the two successive steps \eqref{eq:def_Utilde_V}\eqref{eq:def_V_W} relating the potentials $U,\tilde U, V, W$ gives $W=U +\log \left(\frac{\sqrt{\theta}}{\eta^2}\right)$.
Since $U$ is smooth up to the boundary and $\eta,\t>0$ in the interior this is a smooth function in the interior.
On the left boundary $x\to 0$ (the case $x\to 1$ is identical) we recall that $\eta,\t$ vanish linearly.
Hence as $x\to 0$ we have $W\sim \log\left(x^{-3/2}\right),\t(x)\sim C x$ and therefore
$$
x\to 0:
\qquad
\sqrt{\t}\frac{d}{dx}\left(\sqrt{\t}\frac{d}{dx}W\right)
\sim C\sqrt{x}\frac{d}{dx}\left(\sqrt{x}\frac d{dx}\log\left(x^{-3/2}\right)\right)
=\frac{C'}{x}\to +\infty.
$$
(this formal differentiation can be made completely rigorous by expanding out and keeping track of the leading terms).
Thus this function blows-up to $+\infty$ as $x\to 0^+$, and similarly for $x\to 1^-$.
As a consequence $\inf=\min$ is attained for a finite value at an interior point in \eqref{eq:def_lambda_mod_convexity}.

In order to establish the more difficult displacement convexity, let $D=\operatorname{diam}(\bar E)=\dS(0,1)$ and $F=(0,D)$.
In dimension $d=1$ the arclength map
 $$
 \begin{array}{crcl}
 i: & \bar E & \rightarrow & \bar \D\\
    & x     & \mapsto & \dS(0,x)
    \end{array}
 $$
 gives an isometry between the metric spaces $(\bar E,\dS)$ and $(\bar\D,\dist_\R)$, the latter being equipped with the Euclidean distance.
 This induces in turn an isometry
 $$
 \begin{array}{crcl}
 I: & \P_\t(\bar E) & \rightarrow & \P_\R(\bar\D)\\
    &  q     & \mapsto & i_\# q 
    \end{array}
 $$
 between the Wasserstein spaces $(\P(\bar E),\W_{\t})$ and $(\P(\bar\D),\W_\R)$.
 Moreover, by standard properties of the pushforward operation, it is easy to check that, if
 $$
 \bar q \coloneqq I( q )\coloneqq i_\#  q 
 \qqtext{and}
 \bar\pi\coloneqq I(\pi)\coloneqq i_\# \pi,
 $$
 then the Radon-Nykodim derivative $\bar\sigma(y)\coloneqq \frac{\rd \bar q }{\rd \bar\pi}(y)=\frac{\rd  q }{\rd \pi}(i^{-1}(y))=\sigma(i^{-1}(y))$.
 This gives
 $$
 \bar{\H}_{\bar\pi}( \bar q )\coloneqq
 \int_{\bar\D}\bar\sigma(y)\log\bar\sigma(y)\bar\pi(\rd y) 
 =\int _{\bar E} \sigma(x)\log\sigma(x)\pi(\rd x)
 =
 \H_{\pi}( q ).
 $$
 Since $y=i(x)$ is a bijection so is $\bar q =I( q )$, and because the isometry $I$ maps $\W_\t$-geodesics to $\W_\R$-geodesics, clearly \eqref{eq:convexity_E} will be established if we can show that
 $$
 \bar{\H}_{\bar\pi}(\bar q _t)\leq (1-t)\bar{\H}_{\bar\pi}(\bar q _0)+t\bar{\H}_{\bar\pi}(\bar q _1)-\lambda\frac{t(1-t)}{2}\W_\R^2(\bar q _0,\bar q _1),
 \qquad t\in[0,1]
 $$
 for all $\W_\R$-geodesic $(\bar q _t)_{t\in [0,1]}$ joining arbitrary measures $\bar q _0,\bar q_1\in\P_\R(\bar\D)$.
 The main advantage is that this is now written with respect to the Wasserstein distance based upon $\bar \D=[0,D]$ endowed with the smooth, complete, Euclidean metric.
 This allows the completely rigorous application of \cite[Theorem 1.3]{sturm2005convex}, in the simplest of all settings since there is no Ricci curvature in dimension $d=1$.
 More precisely, if $\bar W(y)\coloneqq cst-\log \bar\pi(y)$ is such that
 $$
 \bar\pi(\rd y)=\frac{1}{\mathcal Z}\exp(-\bar W(y))\rd y
 \qqtext{with}
 \bar\H_{\bar\pi}(\bar q)=\int_{\bar \D}\bar q(y)\log\bar q(y)\,\rd y + \int_{\bar \D}\bar W(y) \bar q(y)\,\rd y-cst,
 $$
 the modulus of displacement-convexity $\lambda$ of $\bar{\H}_{\bar\pi}$ (hence of $\H_\pi$) is given by any modulus of (linear) convexity $\frac{d^2}{dy^2}\bar W(y)\geq \lambda$.
 The point here is that $\rd y$ is now really the volume form on $\bar \D$, and all the pre-isometry issues have now been encoded in the explicit potential $\bar W$.
 By definition of the pushforward $\bar\pi=i\pf \pi$ it is easy to see that
 $\bar\pi(y) =\frac{\pi}{|\partial_x i|}\circ i^{-1}(y)
=[\sqrt{\t}\pi] \circ i^{-1}(y)
 $.
 In particular
 $$
\bar W(y)= cst -\log\bar\pi(y)
=cst-\log(\pi\sqrt{\t})(i^{-1}(y))
=cst-\log(\pi(i^{-1}(y)))-\frac{1}{2}\log\t(i^{-1}(y)).
 $$
 With this new explicit formula at hand and $y=i(x)$, $\frac{d}{dy}=\sqrt{\t(x)}\frac{d}{dx}$, straightforward calculus finally confirms that $\frac{d^2}{dy^2}\bar W(y)\geq \lambda$ is exactly equivalent to \eqref{eq:def_lambda_mod_convexity}.
 \end{proof}
When attempting to naively investigate displacement convexity for the unconditioned model, one would immediately face the potential $\hat U(x)=U(x)+\log\t(x)$ as in section~\ref{sec:Kimura_not_GF}.
The effect of the Shahshahani geometry would induce the very same $-\log\operatorname{vol}(x)=-\frac 12\log\t(x)$ compression (that appeared in the previous proof by isometry), but any putative modulus of displacement convexity would still be given in the end by the modulus of Riemannian convexity of $U+\log \t-\frac 12 \log \t=U+\frac 12 \log \t$.
The $+\log \t$ term blowing down to $-\infty$ on the boundary actually makes this singularly concave as a function of $x$ (even in Riemannian coordinates), and as a result the naive entropy from section~\ref{sec:Kimura_not_GF} would not be $\lambda$-displacement convex for any $\lambda\in \R$.
This is yet another hint that conditioning is really needed.
\\

With convexity now being settled, let us turn to an equally fundamental question, the practical computation of the metric slope w.r.t. the $\W_\t$ metric.
\begin{prop}
\label{prop:Fisher_slope_strong_upper_grad}
 Consider the relative Fisher information
 \begin{equation}
 \label{eq:def_Fisher_Iq}
  \mathcal I_\pi( q )\coloneqq 
  \begin{cases}
   \int_E\t(x)\left|\frac{\nabla\sigma}{\sigma}(x)\right|^2 q (\rd x) & \mbox{if } q(\rd x) =\sigma(x)\pi(\rd x) \in \Dom(\mathcal I_\pi),\\
   +\infty & \mbox{otherwise}
  \end{cases}
 \end{equation}
with domain
\begin{equation}
\Dom(\mathcal I_\pi)
\coloneqq \left\{ q =\sigma\cdot \pi \in D(\mathcal H_\pi) \qtext{s.t.} \sigma\in W^{1,1}_{\mathrm{\mathrm{loc}}}(E),\quad
\frac{\nabla\sigma}{\sigma}\in L^2(\t q )\right\}.
\end{equation}
Then the local slope
\begin{equation}
|\partial\H_\pi|_{\W_\t}( q )=\sqrt{\mathcal I_\pi( q )}
\label{eq:computation_slope=fisher}
\end{equation}
is narrowly lower semi-continuous and provides a strong upper gradient for $\mathcal H_\pi$.
\end{prop}
\begin{proof}
 As before we argue via isometry, and for any $ q \in \P_\t(\bar E)$ we write $\bar q =i_\# q \in \P_\R(\bar\D)$.
 From \eqref{eq:def_metric_slope_abstract} clearly isometry preserves the metric slope
 $$
 |\partial\H_\pi|_{\W_\t}( q )=|\partial\bar\H_{\bar\pi}|_{\W_\R}(\bar q )
 $$
 and it suffices to compute $|\partial\bar\H_{\bar\pi}|_{\W_\R}(\bar q )$.
 In the standard, Euclidean $(\P(\bar\D),\W_\R)$ framework we can rigorously apply \cite[Theorem 10.4.9]{AGS} to conclude that
 $$
 |\partial\bar\H_{\bar\pi}|_{\W_\R}^2(\bar q )=\int_{\D}\left|\frac{\nabla_y \bar\sigma}{\bar\sigma}(y)\right|^2\,\bar q (\rd y),
 \hspace{1cm}\bar q(\rd y)=\bar\sigma(y)\cdot\bar\pi(\rd y)
 $$
 with corresponding domain.
 Here of course we implicitly mean that $\bar q \ll\bar\pi$ with density $\bar\sigma(y)=\frac{\rd\bar q }{\rd\bar\pi}(y)$ belonging to $W^{1,1}_{\mathrm{loc}}(\D)$ and making the right-hand side integral finite.
 Since $\bar q =i_{\#} q $ and $\bar\pi=i_{\#}\pi$ the corresponding densities are again related through 
 $$
 \bar\sigma(y)=\sigma(i^{-1}(y)).
 $$
 In particular, since $i$ and $i^{-1}$ are locally smooth, $\bar\sigma\in W^{1,1}_{\mathrm{loc}}(\D)$ if and only if $\sigma\in W^{1,1}_{\mathrm{loc}}(E)$.
 Moreover, owing to $i(x)=\dS(0,x)=\int_0^x\frac{1}{\sqrt{ \t(z)}}\rd z$ hence $\frac{di}{dx}=\frac{1}{\sqrt{\t(x)}}$, we get
 $$
 \nabla_y\bar\sigma(y)=\frac{di^{-1}}{dy}(y)\nabla_x\sigma(i^{-1}(y))
 =\frac{1}{\frac{di}{dx}(i^{-1}(y))}\nabla_x\sigma(i^{-1}(y))
 =\sqrt{\t(i^{-1}(y))}\nabla_x\sigma(i^{-1}(y)).
 $$
 By definition of the pushforward $\bar q =i_\# q $ this gives
 $$
 |\partial\H_\pi|_{\W_\t}^2( q )
 =|\partial\bar \H_{\bar\pi}|_{\W_\R}^2(\bar q )
 =\int_{\D}\left|\frac{\sqrt{\t}\nabla_x \sigma}{\sigma}(i^{-1}(y))\right|^2\,\bar q (\rd y)
 =\int_{E}\t(x)\left|\frac{\nabla_x \sigma}{\sigma}(x)\right|^2\,q (\rd x)
 $$
 as in our statement.
 
 In order to conclude the proof, we finally recall from Proposition~\ref{prop:geodesic_convexity_alpha} that $\H_\pi$ is $\lambda$-geodesically convex:
\cite[Corollary 2.4.10]{AGS} guarantees that the local slope $|\partial\H_\pi|_{\W_\t}$ coincides with the \emph{global slope}, and as such it is automatically lower-semicontinuous and a \emph{strong}-upper gradient.
\end{proof}

%
%
\section{The conditioned gradient flow}
\label{sec:relate_PDE_metric}
The goal of this section is to identify the (smooth) solutions $q_t(x)$ of the conditioned Fokker-Planck equation \eqref{eq:Fokker-Planck_q_PDE}, constructed in Proposition~\ref{prop:properties_sol_q_PDE} by purely PDE techniques, as $\EVI$ solutions of the abstract gradient flow
$$
``\frac{d}{dt}q_t=-\grad_{\W_\t}\H_\pi(q_t)".
$$

Let us first recall that the potential $V$ from \eqref{eq:def_V_W} was defined so that
$$
\pi(\rd x)=\frac{1}{\mathcal Z}e^{-V(x)}\rd x
\qqtext{and}
\partial_t q_t = \dive\left(q_t\t\nabla\log\left(\frac{q_t}{\pi}\right)\right).
$$
Writing as before
$$
\sigma_t(x)\coloneqq \frac {q_t(x)}{\pi(x)}=\frac{d q_t}{d\pi}(x)
$$
for the density w.r.t. our stationary measure, we can recast the PDE as the continuity equation
$$
\partial_t q_t+\dive(q_t v_t)=0
\qqtext{with}
v_t(x)= -\t(x)\nabla\log\sigma_t(x).
$$
This resembles both the Fisher information \eqref{eq:def_Fisher_Iq} and the metric speed from Theorem~\ref{theo:AC2_curves}.
Indeed we have:
\begin{prop}
 \label{prop:PDE_q_slope_speed}
 Let $q_t(x)$ be the smooth solution from Proposition~\ref{prop:properties_sol_q_PDE} and $\sigma_t(x)=\frac{q_t(x)}{\pi(x)}$.
 Then for all $T>t_0>0$ there exists $C=C(t_0,T)>0$ such that
 \begin{equation}
 \label{eq:bound_nabla_log_sigma}
 |\nabla\log\sigma_t(x)|=\left|\frac{\nabla\sigma_t(x)}{\sigma_t(x)}\right|\leq \frac{C}{\pi(x)},
 \hspace{1cm}\forall x\in E
 \end{equation}
 and
 \begin{equation}
 \label{eq:speed_slope_q}
 |q'|^2_{\W_\t}(t)
 \leq \|v_t\|^2_{L^2_\t(q_t)}=\int_E \frac{|v_t|^2}{\t} q_t
= \int_E \t\left|\frac{\nabla \sigma_t}{\sigma_t}\right|^2 q_t = |\partial \H_\pi|^2(q_t)\leq  C  
 \end{equation}
 for all $t\in [t_0,T]$.
\end{prop}
\begin{proof}
Fix any $0<t_0<T$.
We recall from Proposition~\ref{prop:properties_sol_q_PDE} that $q_t(x),\pi(x)>0$ are smooth up to the boundary, vanish linearly, and that $0<c_0\leq \frac{q_t(x)}{\pi(x)}=\sigma_t(x)\leq C_0$ for $t\geq t_0$.
In particular $\sigma\in C([t_0,+\infty)\times\bar E)\cap C^\infty([t_0,\infty)\times E)$, and the chain rule $\nabla \log\sigma_t=\frac{\nabla\sigma_t}{\sigma_t}$ is completely rigorous at least in the interior.
Recall moreover that $\pi\in C^\infty(\bar E)$ and that $q_t\in C^\infty(\R^+\times\bar E)$.
As a consequence $|\nabla q_t(x)|+|\nabla \pi(x)|\leq C$ for any $(t,x)\in [t_0,T]\times\bar E$.
Writing $\nabla\sigma_t=\nabla\left(\frac{q_t}{\pi}\right)=\frac{\pi\nabla q_t-q_t\nabla\pi}{\pi^2}$ and recalling that $q_t\asymp\pi$ are comparable up to the boundary, we see that $|\nabla\sigma_t(X)|\leq \frac{C}{\pi(x)}$ and our bound \eqref{eq:bound_nabla_log_sigma} for $\nabla\log\sigma_t=\frac{\nabla\sigma_t}{\sigma_t}$ follows from the lower bound $\sigma_t(x)\geq c_0>0$.

Let us now focus on \eqref{eq:speed_slope_q} and relate $\nabla\log \sigma$ to the metric speed as in Proposition~\ref{prop:Fisher_slope_strong_upper_grad}.
Since $\sigma_t(x)$ is smooth and positive in the interior we have of course $\frac{\nabla\sigma_t}{\sigma_t}\in W^{1,1}_{\mathrm{loc}}(E)$.
In order to check the upper bound in \eqref{eq:speed_slope_q} we simply use the previous estimate on $\nabla \log\sigma$ to get
$$
|\partial \H_\pi|^2(q_t)
=\int_E \t\left|\frac{\nabla \sigma_t}{\sigma_t}\right|^2 q_t 
=
\int_E \t\left|\nabla \log\sigma_t\right|^2 q_t 
\leq \int_E\t\frac{C}{\pi^2} q_t \leq C,
$$
where the last inequality follows again from \eqref{eq:bounds_qt} ($q_t,\pi,\t$ being all comparable one to another).
This also entails
$$
\|v_t\|^2_{L^2_\t(q_t)}=\int_E \frac{|v_t|^2}{\t} q_t
= \int_E \t\left|\frac{\nabla \sigma_t}{\sigma_t}\right|^2 q_t \leq C
$$
locally uniformly in time $t\in[t_0,T]$.
Our estimate \eqref{eq:bound_nabla_log_sigma} also controls
\begin{equation}
\label{eq:control_flux=0_boundary}
|q_t(x) v_t(x)|=\left|q_t(x) \t(x) \nabla\log\sigma_t(x)\right|=\mathcal O\left(\pi^2(x)\frac{1}{\pi(x)}\right)\xrightarrow[x\to \partial E]{}0
\end{equation}
locally uniformly in time.
As a consequence the flux $q_t v_t\cdot\nu|_{\partial E}=0$ vanishes in the strong (continuous) sense, and \textit{a fortiori} $(q,v)$ solve the continuity equation in the sense of Definition~\ref{def:weak_sols_CE}.
Our characterization of absolutely continuous curves from Theorem~\ref{theo:AC2_curves} finally guarantees that $q\in AC^2([t_0,T];\P_{\W_\t})$ with $|q'|^2_{\W_t}(t)\leq \|v_t\|^2_{L^2_\t(q_t)}\leq C$ and the proof is complete.
\end{proof}
\begin{cor}
\label{cor:EDI_t0_q}
 The solution $q$ from Proposition~\ref{prop:properties_sol_q_PDE} is an $\EDI_{t_0}$ solution for any $t_0>0$.
\end{cor}
\begin{proof}
We will actually establish a stronger differential version of \eqref{eq:EDI0}.
For this we simply use the PDE to compute, for $t\geq t_0>0$,
\begin{multline*}
 \frac{d}{dt}\H_\pi(q_t)
 =\frac{d}{dt}\int_E\sigma_t\log\sigma_t\,\pi
 =\int_E(1+\log\sigma_t)\partial_t\sigma_t\,\pi
 =\int_E(1+\log\sigma_t)\partial_tq_t
 \\
 =\int_E(1+\log\sigma_t)\dive\left(\t q_t\nabla \log\sigma_t\right)
 = - \int_E\t |\nabla \log\sigma_t|^2\,q_t.
\end{multline*}
The first differentiation w.r.t. time under the integral sign is completely rigorous since $q,\sigma$ are smooth for $t>0$, and $0<c\leq \sigma_t(x)\leq C$ from \eqref{eq:bounds_qt}.
Our last integration by parts in space is also legitimate, since as already observed from \eqref{eq:control_flux=0_boundary} $q_t\t\nabla\log\sigma_t=0$ in the boundary term
$$
\int_{\partial E}(1+\log\sigma_t) \t q_t\nabla\log\sigma_t\cdot\nu =0
$$
($\log\sigma$ being bounded uniformly in $[t_0,T]\times\bar E$).
Finally leveraging \eqref{eq:speed_slope_q} from the previous Proposition, this yields
\begin{multline*}
\frac{d}{dt}\H_\pi(q_t) = -\frac{1}{2}\int_E\t |\nabla \log\sigma_t|^2q_t -\frac{1}{2}\int_E\t |\nabla \log\sigma_t|^2q_t\\
= -\frac{1}{2}\underbrace{\int_E\frac{1}{\t} |v_t|^2 q_t}_{\geq |q'|^2_{\W_\t}(t)} -\frac{1}{2}\underbrace{\int_E\t \left|\frac{\nabla \sigma_t}{\sigma_t}\right|^2q_t}_{=|\partial \H_\pi|^2(q_t)}
\leq -\frac{1}{2}|q'|^2_{\W_\t}(t) -\frac{1}{2}|\partial \H_\pi|^2(q_t).
\end{multline*}
and \eqref{eq:EDI0} immediately follows.
\end{proof}
We are now in position of proving our main result.
\begin{theo}
\label{theo:q_EVI}
Let $\lambda\in\R$ be as in \eqref{eq:def_lambda_mod_convexity}.
Then
\begin{enumerate}
 \item 
 \label{item:Hpi_generates_EVI}
 $\H_\pi$ generates an $\EVI_\lambda$-gradient flow $(\mathsf S_t)_{t\geq 0}$ on $\P_\t(\bar E)=\overline{\Dom (\H_\pi)}$
 \item
 \label{item:q=EVI}
 For any $q_0\in \P(\bar E)$ emanating from some $p_0\in \P(\bar E)$ as in \eqref{eq:def_q_0_from_p0}, the solution $q_t(x)$ from Proposition~\ref{prop:properties_sol_q_PDE} coincides with the $\EVI$ solution, i-e
 $$
 q_t=\mathsf S_t q_0,
 \hspace{1cm}t\geq 0.
 $$
 \item
 For any $q_0^1,q_0^2\in \P(\bar E)$ we have contraction in the Wasserstein-Shahshahani distance,
 $$
 \W_\t(q^1_t,q^2_t)\leq e^{-\lambda t}\W_\t(q^1_0,q^2_0),
 \hspace{1cm}t\geq 0.
 $$
 \item
 Long time behaviour and logarithmic Sobolev inequality
 \begin{enumerate}
 \item
 If $\lambda>0$ then
 \begin{equation}
  \label{eq:LSI}
  \int_{E}\sigma\log\sigma \,\pi= \H_{\pi}( q )
 \leq\frac{1}{2\lambda}\mathcal I_\pi( q )
 = \frac{1}{2\lambda}\int_E\t\left|\frac{\nabla\sigma}{\sigma}\right|^2  q ,
 \tag{$LSI_\lambda$}
 \end{equation}
for all $ q =\sigma\cdot \pi\in D(\mathcal H_\pi)$, with moreover
\begin{equation}
\label{eq:long_time_lambda>0}
\W_\t(q_t,\pi)\leq e^{-\lambda t}\W_\t(q_0,\pi)
\qqtext{and}
\|q_t-\pi\|_{\mathrm{TV}}\leq e^{-\lambda t}\sqrt{2\H_\pi(q_0)}
\end{equation}
for all $t\geq 0$.
In particular $q_t$ converges exponentially fast both in the Wasserstein distance and in total variation to the stationary distribution (provided $\H_\pi(q_0)>0$, but otherwise the argument can be applied starting from $\H_\pi(q_{t_0})<\infty$ for small $t_0>0$).
 \item
 If $\lambda=0$ then
 $$
 \W_{\theta}(q_t,\pi)\xrightarrow[t\to\infty]{} 0.
 $$
 \end{enumerate}
\end{enumerate}
\end{theo}
\noindent
In our opinion the main point here is really \ref{item:q=EVI}:
Although the unconditioned process is \emph{not} variational (as discussed in Section~\ref{sec:Kimura_not_GF}), the $Q$-process  automatically inherits after conditioning a (rigorous) variational structure, once the right geometry is used (that of $\W_\t$).
This significantly strengthens previous results from \cite{chalub2021gradient}, where the Wasserstein-Shahshahani gradient flow structure was identified only formally for the original Kimura PDE \eqref{eq:FP_with_BC} satisfied by $p$ and without any real probabilistic interpretation of the change of variables from $p$ to $q$.

It is not clear what the biological interpretation might be for initial data $q_0$ in \ref{item:Hpi_generates_EVI} at least partially supported on the boundary.
Indeed, in our probabilistic construction \eqref{eq:def_q_0_from_p0} prevented $q_0$ from having atoms on the boundary (and rightly so, since $q_0$ precisely represents $p_0$ conditioned to non absorption).
We still chose to include this possibility for the sake of generality.

 As regards \eqref{eq:LSI}, we point out that very similar weighted logarithmic Sobolev inequalities were obtained independently in \cite{furioli2019wright} from a purely PDE and less geometric point of view (their proof is actually a hands-on change of variables, corresponding to our isometry argument in Section~\ref{subsec:convexity_slope}).

\begin{proof}
~
\begin{enumerate}
 \item 
 We already proved in Proposition~\ref{prop:geodesic_convexity_alpha} that $\H_\pi$ is $\lambda$-geodesically convex.
 Since in dimension $d=1$ there is no Ricci curvature, it is well-known that $(\P_\t(\bar E),\W_\t)$ is a non-positively curved $\mathsf{CAT}(0)$ space in the Alexandrov sense.
 As a consequence the generation of the $\EVI_\lambda$-flow can be found verbatim in \cite[Theorem 3.14]{muratori2020gradient}, see also \cite[Theorem 4.0.4]{AGS} (in which Assumption 4.0.1 holds along geodesics in dimension 1 only).
 \item
 We just established in Corollary~\ref{cor:EDI_t0_q} that, if $q_0$ originates from some $p_0$ in the sense of \eqref{eq:def_q_0_from_p0}, the solution $q_t(x)$ obtained by pure PDE methods is actually an $\EVI_{t_0}$-solution for any $t_0>0$.
 Since $q_t\to q_0$ narrowly (Proposition~\ref{prop:properties_sol_q_PDE}), and given that we just proved that $\H_\pi$ generates an $\EVI_\lambda$-flow on the whole $\P_\t(\bar E)$, we can simply appeal to Theorem~\ref{theo:EDIt0_EVI} and conclude that $q_t=\mathsf S_t q_0$.
\item
This is one of the classical properties of $\EVI$ solutions in Theorem~\ref{theo:prop_EVI_sols}.
\item
For $\lambda>0$ the statement (in particular the exponential convergence and the logarithmic Sobolev inequality) can be found e.g. in \cite[Theorem 3.5]{muratori2020gradient} and \cite[Lemma 2.4.13]{AGS} (Assumption 2.4.5 and Assumption 4.0.1 therein are both satisfied in our one-dimensional $\mathsf{CAT}(0)$ setting).
For the asymptotic convergence for $\lambda=0$ the result can be found in \cite[Corollary 4.0.6]{AGS}.
The convergence in Total variation comes from the celebrated Csisz\'ar-Kullback-Pinsker inequality $\|q-\pi\|_{\mathrm{TV}}\leq \sqrt{2\H_{\pi}(q)}$ as well as the decay in entropy $\H_\pi(q_t)\leq e^{-2\lambda t}\H_\pi(q_0)$.
\end{enumerate}

\end{proof}
Let us finish with a brief discussion on convergence rates.
Note first that no explicit rate of convergence can be derived for $\lambda=0$.
In the better case $\lambda>0$, one may wonder what additional information \eqref{eq:long_time_lambda>0} brings, compared to any possible spectral analysis one may try carrying out directly for $\tilde \L$.
From \eqref{eq:def_generator_Ltilde} clearly the spectrum of $\tilde \L$ is shifted by $-\lambda_0$ w.r.t. that of $\L$, and one should expect that a spectral analysis for $\tilde \L$ leads to exponential convergence $q_t\to \pi$ with rate $\lambda_1-\lambda_0>0$ (see also the proof of Theorem~\ref{theo:exist_alpha} for a hint in this direction).
However, already in the Euclidean case it is known that the modulus of displacement convexity $\lambda<\lambda_1-\lambda_0$ predicts a slower convergence rate.
In other words, our analysis should predict suboptimal rates with respect to more direct spectral considerations.
This can be confirmed in the prototypical case $U\equiv 0 $ (neutral genetic preference) with $\t(x)=x(1-x)$, where one can check that $\alpha_0(x)=1$ and $\alpha_1(x)=x-\frac 12$ are the first eigenfunctions of $\Delta(x(1-x)\,\cdot\,)$ with eigenvalues $\lambda_0=6,\lambda_1=2$ and gap $\lambda_1-\lambda_0=4$, while our best predicted rate in this case can be tracked down to $\lambda=3$ by explicitly evaluating \eqref{eq:def_lambda_mod_convexity}.
We are not aware of a general statement that $\lambda<\lambda_1-\lambda_0$ always, in particular the $\t$-degeneracy in our framework does not fit the usual irreducible setting and might induce new behaviours.
Despite this suboptimality, one possible advantage of our approach is that the long-time convergence towards the stationary distribution $q_t\to \pi$ can be explicitly quantified by accessing only the lowest spectral level $\lambda_0$ of the unconditioned operator $\L$, which may be interesting for applications in biology.

\section*{Acknowledgments}
J.-B.C. was supported by FCT - Fundação para a Ci\^encia e a Tecnologia, under the project UIDB/04561/2020
L.M. was funded by the Portuguese Science Foundation through a personal grant 2020/00162/CEECIND as well as the FCT project PTDC/MAT-STA/28812/2017, and wishes to thank F. Chalub and N. Champagnat for fruitful discussions.
\bibliographystyle{plain}
\bibliography{./biblio}

\bigskip
\noindent
{\sc
Jean-Baptiste (\href{mailto:jeanbaptiste.casteras@gmail.com}{\tt jeanbaptiste.casteras@gmail.com}).
\\
CMAFcIO, Faculdade de Ci\^encias da Universidade de Lisboa, Edificio C6, Piso 1, Campo Grande 1749-016 Lisboa, Portugal
}
\\

\noindent
{\sc
L\'eonard Monsaingeon (\href{mailto:leonard.monsaingeon@univ-lorraine.fr}{\tt leonard.monsaingeon@univ-lorraine.fr}).
\\
Institut \'Elie Cartan de Lorraine, Universit\'e de Lorraine, Site de Nancy B.P. 70239, F-54506 Vandoeuvre-l\`es-Nancy Cedex, France
\\
Grupo de F\'isica Matem\'atica, GFMUL, Faculdade de Ci\^encias, Universidade de Lisboa, 1749-016 Lisbon, Portugal.
}

\end{document}